\documentclass{amsart}

\RequirePackage{doi}
\usepackage{hyperref}

\usepackage{amsmath}
\usepackage{tikz-cd}
\usepackage{amsthm}
\usepackage{hyperref}
\usepackage{amsfonts,graphics,amsthm,amsfonts,amscd,latexsym}
\usepackage{epsfig}
\usepackage{flafter}
\usepackage{mathtools}
\usepackage{comment}
\usepackage{stmaryrd}
\usepackage{cite}

\usepackage{amssymb}

\hypersetup{
    colorlinks=true,    
    linkcolor=blue,          
    citecolor=blue,      
    filecolor=blue,      
    urlcolor=blue           
}
\usepackage{tikz}
\usetikzlibrary{graphs,positioning,arrows,shapes.misc,decorations.pathmorphing}

\tikzset{
    >=stealth,
    every picture/.style={thick},
    graphs/every graph/.style={empty nodes},
}

\tikzstyle{vertex}=[
    draw,
    circle,
    fill=black,
    inner sep=1pt,
    minimum width=5pt,
]
\usepackage[position=top]{subfig}
\usepackage{amssymb}
\usepackage{color}

\calclayout

\usetikzlibrary{decorations.pathmorphing}
\tikzstyle{printersafe}=[decoration={snake,amplitude=0pt}]

\newcommand{\qq}{\mathbb{Q}}
\newcommand{\zz}{\mathbb{Z}}

\newcommand{\rr}{\mathbb{R}}

\newcommand{\kk}{\mathbb{K}}

\def\O#1.{\mathcal {O}_{#1}}			
\def\pr #1.{\mathbb P^{#1}}				
\def\af #1.{\mathbb A^{#1}}			
\def\ses#1.#2.#3.{0\to #1\to #2\to #3 \to 0}	
\def\xrar#1.{\xrightarrow{#1}}			
\def\K#1.{K_{#1}}						
\def\bA#1.{\mathbf{A}_{#1}}			
\def\bM#1.{\mathbf{M}_{#1}}				
\def\bL#1.{\mathbf{L}_{#1}}				
\def\bB#1.{\mathbf{B}_{#1}}				
\def\bK#1.{\mathbf{K}_{#1}}			
\def\subs#1.{_{#1}}					
\def\sups#1.{^{#1}}

\DeclareMathOperator{\mld}{mld}

\usepackage{tikz}
\usetikzlibrary{matrix,arrows,decorations.pathmorphing}

  \newtheorem{introthm}{Theorem}

  \newtheorem{introconj}{Conjecture}

  \newtheorem{theorem}{Theorem}[section]
  \newtheorem{lemma}[theorem]{Lemma}
  \newtheorem{conj}[theorem]{Conjecture}
  \newtheorem{proposition}[theorem]{Proposition}

  \newtheorem{notation}[theorem]{Notation}
  \newtheorem{definition}[theorem]{Definition}
  \newtheorem{example}[theorem]{Example}

\newtheorem{remark}[theorem]{Remark}

\theoremstyle{remark}

\numberwithin{equation}{section}

\usepackage[all]{xy}

\begin{document}

\title[On termination and fundamental groups]
{On termination of flips and fundamental groups}

\author[J.~Moraga]{Joaqu\'in Moraga}
\address{Department of Mathematics, Princeton University, Fine Hall, Washington Road, Princeton, NJ 08544-1000, USA
}
\email{jmoraga@princeton.edu}

\subjclass[2010]{Primary 14E30, 
Secondary 14M25.}
\maketitle

\begin{abstract}
In this article, we propose a boundedness conjecture for the
regional fundamental group of klt singularities. 
We prove that this boundedness conjecture, the Zariski closedness of the diminished base locus of $K_X$, and an upper bound for minimal log discrepancies, 
imply the termination of flips with scaling.
We prove the boundedness conjecture 
in the case of toric singularities, quotient singularites, isolated 
$3$-fold singularities, and 
exceptional singularities.
\end{abstract}

\setcounter{tocdepth}{1} 
\tableofcontents

\setcounter{tocdepth}{1}
\tableofcontents

\section{Introduction}

The study of singularities through
the lens of topology date back to Mumford, Milnor, and Durfee~\cite{Mum61,Mil68,Dur83}.
In~\cite{Mum61}, Mumford established that the smoothness
of a surface singularity can be characterized 
by its fundamental group.
This property fails for higher dimensional singularities.
Moreover, any finitely presented group can 
be achieved as the local fundamental group
of an algebraic singularity~\cite{Kol11,Kol13b,KK14}.
Although this result seems to indicate that the study of the topology of algebraic singularities is hopeless,
many positive results arise when we specialize
to the right class of singularities.
For birational geometers, klt singularities are such category of singularities.
On one hand, this class is wide enough so that the minimal model smooth varieties have klt singularities.
On the other hand, klt singularities are mild enough so that most vanishing theorems for smooth varieties 
still hold for varieties with klt singularities.
This makes the understanding of the geometry and topology of klt singularities a central topic in birational geometry.

The development of the minimal model program
has been intertwined with a better understanding
of klt singularities. 
There have been some recent developments 
on the topology
of the singularities of the MMP.
We summarize some of these results.
In~\cite{Kol11,Kol13}, Koll\'ar asks whether the local fundamental group of a klt singularity is finite.
In~\cite{Xu14}, Xu proved that the \'etale fundamental group of a klt singularity is finite.
This statement is also proved by Greb, Kebekus, and Peternell~\cite{GKP16} partially using the argument in~\cite{Xu14}.
In~\cite{TX17}, Tian and Xu reduce the previous conjecture
to the finiteness of the fundamental group
of the smooth locus of a Fano type variety.
This conjecture was answered positively by Braun in~\cite{Bra20}.
More generally, in~\cite{Bra20}, Braun proves the finiteness of the {\it regional fundamental group} of the singularity, i.e., the fundamental group of the smooth locus of an analytic neighborhood of the singularity.
Afterward in~\cite{BFMS20}, Braun, Filipazzi, Svaldi, and the author proved that the fundamental group of a klt singularity
satisfies the Jordan property, i.e., it
is almost a finite abelian group of rank at most the dimension of the germ.
The main theorem of~\cite{BFMS20} says that the regional fundamental groups of $n$-dimensional klt singularities are 
the same as the regional fundamental groups
of $n$-dimensional quotient singularities.
In~\cite{XZ20}, Xu and Zhuang proved that the inverse of the normalized volume of a klt singularity gives an upper bound for the order of its regional fundamental group.
In~\cite{BM21}, Braun and the author showed a connection between the Cox ring of Fano type varieties and the fundamental group of klt singularities.
In~\cite{Mor21a}, the author proves that the isomorphism in~\cite{BFMS20} can be realized geometrically by performing certain blow-up at the singularity. 
This generalized the main theorem for toric singularities in~\cite{Mor20c}.
Many of the ideas of~\cite{Mor21a} build into previous work on the dual complex of singularities~\cite{KX16,dFKX17,Bir20,FS20}.

On a different front, there has been some progress on the understanding of minimal log discrepancies; an important invariant of singularities.
The minimal log discrepancies of exceptional singularities are described in the works~\cite{Mor18b,HLS19}.
Exceptional singularities are higher-dimensional analogs of the $E_6,E_7$, and $E_8$ singularities.
By~\cite{Mor18c,HLM20}, these singularities are deformation of cones over exceptional Fano varieties.
Recently, in~\cite{Mor21b}, the author showed the existence of an upper bound for the minimal log discrepancy of klt singularities of regularity one.
We expect the techniques of~\cite{Mor21b} to allow us to understand minimal log discrepancies in a more general setting.

The aim of this article 
is to draw a connection between the topology of the singularity and its minimal log discrepancy.
Shokurov conjectured that klt singularities with the same mld have similar indices of their canonical divisor.
This is known as the index conjecture.
This conjecture is proved by Ambro in the case of toric singularities in~\cite{Amb09}.
The index conjecture and the result in~\cite{Mor21a} suggest the existence of a deeper connection between minimal log discrepancies and the topology of the singularity.
We propose the following conjecture, which asserts the subsequent. Once we control the minimal log discrepancy away from zero and the number of divisorial valuations hitting near the mld, we also control the regional fundamental group.

\begin{introconj}\label{conj:boundedness}
Let $n$ and $N$ be two positive integers.
Let $\epsilon$ and $\delta\in (0,1)$ be two positive real numbers.
There exists a constant $\rho:=\rho(n,N,\epsilon,\delta)$,
only depending on the variables $n,N,\epsilon$ and $\delta$,
satisfying the following.
Let $(X,\Delta;x)$ be a $n$-dimensional klt singularity such that:
\begin{enumerate}
    \item ${\rm mld}(X,\Delta;x)>\epsilon$, and 
    \item there are at most $N$ prime divisors $E$ over $X$ with $c_X(E)=x$ for which 
    \[
    a_E(X,\Delta;x) \in [{\rm mld}(X,\Delta;x),{\rm mld}(X,\Delta;x)+\delta).
    \]
\end{enumerate}
Then, $|\pi_1^{\rm reg}(X,\Delta;x)|\leq \rho$.
\end{introconj}

We will call the previous conjecture; 
the boundedness conjecture of the regional fundamental group.
In Section~\ref{sec:ex}, we give examples that show 
that all the hypotheses of Conjecture~\ref{conj:boundedness} are necessary
to bound the regional fundamental group.
Our first theorem is a positive answer for the
boundedness conjecture in the setting of quotient singularities,
toric singularities, 
isolated $3$-fold singularities,
and exceptional singularities.
These classes of singularities 
are natural from the perspective
of the minimal model program.

\begin{introthm}\label{introthm:boundedness}
The boundedness conjecture of the regional fundamental group,
Conjecture~\ref{conj:boundedness}, holds for the following clases of singularities: 
\begin{enumerate}
    \item Quotient singularities, 
    \item toric singularities, 
    \item isolated $3$-fold singularities whenever $\delta>1-\epsilon$, and
    \item exceptional singularities.
\end{enumerate}
\end{introthm}

The importance of the minimal log discrepancy relies on its connection
to the termination of flips.
In~\cite{Sho04}, Shokurov proposed a conjecture regarding the ascending chain condition for minimal log discrepancies. 
This conjecture is known as the ascending chain condition for mld's.
We abbreviate it as ACC for mld's.
In~\cite{Amb99}, Ambro conjectured that minimal log discrepancies are lower semicontinuous.
This conjecture is known as lower semicontinuity for mld's.
Abbreviated as LSC for mld's.
These two conjectures imply the termination of flips~\cite{Sho04}.
Many theorems on termination of flips reflect this philosophy~\cite{Bir07,Mor18a,HM20}.
The approach is often the same: 
We use some gadgets to measure the singularities of the sequence of flips.
This gadget is either an effective divisor or a moduli divisor.
We use adjunction theory to prove that the MMP eventually terminates along the most singular loci, i.e., the loci along which the threshold is attained.
Adjunction plays the role of lower semicontinuity~\cite{Hac14}. 
Then, we use an ascending chain condition for thresholds to prove that 
the sequence must terminate.
The ACC for thresholds plays the role of the ACC for minimal log discrepancies~\cite{HMX14,BZ16}.
Both conjectures, the ACC and LSC, independently imply the
existence of an upper bound for the minimal log discrepancy
of $n$-dimesional singularities.

\begin{introconj}\label{conj:upp-mld}
Let $n$ be a positive integer. 
There exists a constant $A:=A(n)$,
only depending on $n$, 
that satisfies the following.
Let $(X;x)$ be a $n$-dimensional klt singularity.
Then, there exists a prime divisor $E$ over $X$ with center $x\in X$ for which
$a_E(X;x)<A$.
\end{introconj}

On the other hand, when we run a minimal model program for a pseudo-effective pair $(X,\Delta)$, 
the locus that we flip or contract is always contained in the diminished base locus of $K_X+\Delta$.
Hence, if the diminished base locus of $K_X+\Delta$ is not a Zariski closed subset of $X$, 
then the MMP will not terminate.
Thus, the following conjecture is natural and standard in the minimal model program.

\begin{introconj}\label{conj:dim-base}
Let $(X,\Delta)$ be a klt pair.
Then, the diminished base locus
${\rm Bs}_{-}(K_X+\Delta)$ is Zariski closed.
\end{introconj}

Our second theorem 
states that the boundedness of the regional fundamental group,
the upper bound for minimal log discrepancies,
and the Zariski closedness of the diminished base locus, imply the termination of flips with scaling.

\begin{introthm}\label{thm:term}
Assume that the following statements hold:
\begin{enumerate}
    \item The boundedness of the regional fundamental group, Conjecture~\ref{conj:boundedness}, holds in dimension $n$
    \item the upper bound for the minimal log discrepancy, Conjecture~\ref{conj:upp-mld}, holds in dimension $n$, and 
    \item the Zariski closedness of the diminished base locus, Conjecture~\ref{conj:dim-base}, holds in dimension $n$.
\end{enumerate}
Then, any minimal model program with scaling of an ample divisor terminates in dimension $n$.
\end{introthm}

We recall that the minimal model program with scaling is introduced in~\cite{BCHM10,Sho09}. This is a special kind of MMP which flips and contract divisors 
more effectively.
In particular, we know that in the MMP with scaling, any component of ${\rm Bs}_{-}(K_X+\Delta)$ will be eventually contracted or flipped (see Lemma~\ref{lem:MMP-vs-dim}).
The proof of Theorem~\ref{thm:term}
is similar, in philosophy, to such of~\cite{Sho04}.
Instead of trying to control the minimal log discrepancy,
we will use Conjecture~\ref{conj:boundedness} and Conjecture~\ref{conj:dim-base} to control the regional fundamental groups of the singularities which we encounter in the sequence of flips.
Then, we will use the control of the regional fundamental group and Conjecture~\ref{conj:upp-mld} to say that the minimal log discrepancies can only take finitely many possible values.
To conclude the argument, we use Conjecture~\ref{conj:dim-base} as a replacement for the lower semicontinuity of minimal log discrepancies.

Our third theorem states that Conjecture~\ref{conj:boundedness} itself suffices to prove termination of flips in dimension four.

\begin{introthm}
Assume that the boundedness of the regional fundamental group, Conjecture~\ref{conj:boundedness}, holds in dimension $4$.
Then, any sequence of flips terminates
in dimension $4$.
\end{introthm}

We finish the introduction by
making a parallel between the termination of flips
and the resolution of singularities.
An important step towards the
resolution of singularities
is the existence of local uniformizations.
After achieving such result, 
one needs to {\it glue} these local uniformizations
together to obtain a resolution of singularities.
We expect a similar behavior for the termination of flips.
Meaning that Conjecture~\ref{conj:boundedness} and Conjecture~\ref{conj:upp-mld} should help to prove a local uniformization of flips (over an arbitrary valuation).
While Conjecture~\ref{conj:dim-base} should help to show that such uniformizations can be glued together to obtain the termination.

\subsection*{Acknowledgements} 
The author would like to thank 
Lukas Braun,
Christopher Hacon, Jingjun Han, and J\'anos Koll\'ar for many useful comments. 

\section{Preliminaries} 

In this section,
we recall some preliminaries which will be used in this article:
generalized pairs, singularities of the minimal model program, 
and toric singularities. 
Throughout this article, we work over an algebraically closed field $\kk$
of characteristic zero. 
All the considered varieties are assumed to be normal and quasi-projective unless otherwise stated. 
All the considered divisors are assumed to be $\qq$-divisors.

\subsection{Generalized pairs} 
In this subsection, we recall the basic definition of generalized pairs and discrepancies.
For the theory of b-divisors, we refer the reader to~\cite[Chapter 1]{Cor07}.

\begin{definition}
{\em 
A {\em generalized pair} $(X,\Delta+M)$ is a triple, 
where $X$ is a normal quasi-projective variety, 
$\Delta$ is an effective divisor on $X$, and
$M$ is a b-nef Cartier divisor on $X$, so that
$K_X+\Delta+M$ is $\qq$-Cartier. 
A {\em generalized singularity} is a generalized pair
together with a closed point $x\in X$. 
We usually denote a generalized singularity by 
$(X,\Delta+M;x)$.
When we say that a statement holds for a generalized singularity
$(X,\Delta+M;x)$, we mean that it holds for certain neighborhood
of $X$ around $x$.
}
\end{definition}

\begin{definition}
{\em 
Let $(X,\Delta+M)$ be a generalized pair.
Let $\pi\colon Y\rightarrow X$ be a projective birational morphism
from a normal quasi-projective variety.
Let $E\subset Y$ be a prime divisor.
We denote by $M_Y$ the trace of the b-divisor $M$ on $Y$.
Then, we can write 
\[
\pi^*(K_X+\Delta+M)=K_Y+\Delta_Y+M_Y,
\] 
for a divisor $\Delta_Y$.
The {\em generalized log discrepancy} of $(X,\Delta+M)$ at $E$,
denoted by $a_E(X,\Delta+M)$, is defined to be
\[
a_E(X,\Delta+M)= 1-{\rm coeff}_E(\Delta_Y).
\] 
We say that $(X,\Delta+M)$ is {\em $\epsilon$-generalized log canonical} 
if all its generalized log discrepancies are at least $\epsilon$.
We will write $\epsilon$-glc for short.
We say that $(X,\Delta+M)$ is {\em $\epsilon$-generalized Kawamata log terminal}
if all its generalized log discrepancies are larger than $\epsilon$.
We will write $\epsilon$-gklt for short.
If $\epsilon$ is zero, then we drop it from the notation.
As it is usual, the condition of being glc or gklt can be checked on a log resolution of $(X,\Delta+M)$ where the divisor $M$ descends.
We denote by $c_X(E)$, the center of $E$ on $X$, i.e., the image of $E$ via the projective birational morphism $\pi$. 
We may say log discrepancies, instead of generalized log discrepancies, when the generalized pair structure is clear from the context.
}
\end{definition}

\subsection{Minimal Model Program} 

In this subsection, we recall some lemmas about the minimal model program
for generalized pairs. 
For the basic definitions of flips, we refer the reader to~\cite{HM20}.
The following lemma is known as the monotonicity of log discrepancies~\cite[Proposition 1.12]{HM20}.

\begin{lemma}\label{lem:gen-monotonicity}
Let $(X,\Delta+M)$ be a projective generalized pair. 
Let $\pi\colon (X,\Delta+M)\dashrightarrow (X_1,\Delta_1+M_1)$ be a step of a minimal model program for the generalized pair $(X,\Delta+M)$.
Let $E$ be a prime divisor over $X$. 
Then, we have that 
\[
a_E(X,\Delta+M)\leq a_E(X_1,\Delta_1+M_1).
\]
Furthermore, if $c_E(X)$ is contained in the 
exceptional locus of $\pi$, then we have that 
\[
a_E(X,\Delta+M)< a_E(X_1,\Delta_1+M_1).
\] 
\end{lemma}

\begin{definition}\label{def:dim} 
{\em 
Let $X$ be a normal quasi-projective variety.
Let $D$ be a divisor on $X$.
The {\em diminished base locus} of $D$, denoted by ${\rm Bs}_{-}(D)$, is defined to be
\[
\bigcup_{\lambda>0} {\rm Bs}(D+\lambda A), 
\]
where $A$ is an ample divisor.
The definition is independent of the chosen ample divisor.
The diminished base locus of $X$ is trivial if and only if $D$ is nef.
The diminished base locus of $X$ is a proper subset of $X$ if and only if $D$ is a pseudo-effective divisor.
In general, the diminished base locus can be a countable union of Zariski closed subsets of $X$ (see, e.g.,~\cite{Les14}).
See~\cite{ELMNP06}, for results regarding the diminished base locus.
}
\end{definition}

The following proposition states that every step of the minimal model program 
must happen along the subset ${\rm Bs}_{-}(K_X+\Delta+M)$ (see, e.g.,~\cite[Lemma 1.29]{Mor18a}).

\begin{proposition}\label{prop:flips-vs-dim-base} 
Let $(X,\Delta+M)$ be a projective generalized pair.
Let $\pi\colon (X,\Delta+M)\dashrightarrow (X_1,\Delta_1+M_1)$
be a step of a minimal model program for $(X,\Delta+M)$.
Then, we have that 
\[
{\rm Bs}_{-}(K_{X_1}+\Delta_1+M_1) \subseteq  
\pi_*({\rm Bs}_{-}(K_X+\Delta+M)) \cup {\rm Ex}(\pi^{-1}).
\] 
\end{proposition}

For the MMP with scaling of an ample divisor,
we refer the reader to~\cite[\S 4]{BZ16}. 
Throughout the article, whenever we say MMP with scaling, we mean MMP with scaling of an ample divisor.
The following lemma asserts that in the MMP with scaling,
every component of the diminished base locus is eventually flipped or contracted.

\begin{lemma}\label{lem:MMP-vs-dim}
Let $(X,\Delta+M)$ be a projective generalized klt pair.
Assume that $K_X+\Delta+M$ is pseudo-effective.
Let 
\[
(X,\Delta+M) \dashrightarrow 
(X_1,\Delta_1+M_1) \dashrightarrow 
\dots 
\dashrightarrow 
(X_i,\Delta_i+M_i) \dashrightarrow
\cdots 
\] 
be the steps of a minimal model program for $(X,\Delta+M)$
with scaling of an ample divisor. 
Then, for every component $Z\subseteq {\rm Bs}_{-}(K_X+\Delta+M)$,
we have that $Z\subseteq {\rm Ex}(X\dashrightarrow X_i)$ for $i$ large enough.
\end{lemma}

\begin{proof}
Let $A$ be the ample divisor which is defining the minimal model program with scaling.
For each model $X_i$, we denote by $A_i$ the strict transform of $A$ on $X_i$.
For each step 
\[
\pi_i\colon 
(X_i,\Delta_i+M_i)\dashrightarrow
(X_{i+1},\Delta_{i+1}+M_{i+1})
\] 
of the minimal model program with scaling,
we can associate $\lambda_i>0$ for which 
$\pi_i$ is a $(K_{X_i}+\Delta_i+M_i+\lambda_i A_i)$-flop.
First, asssume that $\lim_{i\rightarrow \infty} \lambda_i = \lambda_\infty >0$.
Then, by~\cite[Lemma 4.4]{BZ16}, we conclude that the minimal model program must terminate.
Then, the statement follows from Definition~\ref{def:dim} and Proposition~\ref{prop:flips-vs-dim-base}.
On the other hand, assume that $\lim_{i\rightarrow \infty}\lambda_i =0$.
Let $Z\subseteq {\rm Bs}_{-}(K_X+\Delta+M)$ be an irreducible component.
Then, we have that 
\[
Z\subseteq {\rm Bs}(K_X+\Delta+M+\mu A), 
\]
for certain $\mu>0$.
We can find $i$ large enough so that $\lambda_i \leq \mu < \lambda_{i-1}$.
In this case, $(X_i,\Delta_i+M_i+\mu A_i)$ is a minimal model for
$(X,\Delta+M+\mu A)$.
In particular, we have that 
\[
{\rm Bs}_{-}(K_{X_i}+\Delta_i+M_i+\mu A_i) = \emptyset.
\] 
By Proposition~\ref{prop:flips-vs-dim-base}, we conclude that $Z$ must be contained in the exceptional locus of $X\dashrightarrow X_i$.
\end{proof}

\begin{definition}\label{def:mld}
{\em 
Let $(X,\Delta+M)$ be a generalized pair.
Let $W\subseteq X$ be a Zariski closed subset. 
The {\em minimal log discrepancy} of $(X,\Delta+M)$ along $W$ is defined to be
\[
{\rm mld}(X,\Delta+M;W):=
\min\{ a_E(X,\Delta+M) \mid c_X(E) \subset W \}. 
\]
For instance, ${\mld}(X,B+M;x)$ denotes the minimal log discrepancy at the closed point of the generalized singularity.
}
\end{definition}

The following lemma shows that the minimal log discrepancy is actually a minimum and there are only finitely many divisors that attain values close to it. 

\begin{lemma}\label{lem:finiteness-e+d}
Let $(X,\Delta+M)$ be a generalized klt pair
and $W\subset X$ be a Zariski closed subset. 
Let $m:={\rm mld}(X,\Delta+M;W)$. 
Let $\delta \in (0,1)$.
Then, there are finitely many exceptional divisors with center contained in $W$ 
and log discrepancy in the interval
$[m,m+\delta)$.
\end{lemma}

\begin{proof}
Let $\pi\colon Y\rightarrow X$ be a log resolution of $(X,\Delta+M)$ on which the b-divisor $M$ descends. 
We may assume that $\pi$ has purely divisorial exceptional locus.
By further blowing-up, we may assume that the following conditions are satisfied:
\begin{enumerate}
\item $\pi^{-1}(W)$ is purely divisorial, 
\item the divisor $\pi^{-1}(W)+\pi^{-1}_*\Delta+ {\rm Ex}(\pi)$
has simple normal crossing support, and 
\item any two prime divisors in $Y$ with log discrepancy in the interval $(0,1)$ with respect to $(X,\Delta+M)$ must be disjoint. 
\end{enumerate} 
The first two conditions are standard.
The third condition can be achieved as in~\cite[Proposition 2.36]{KM98}.
Write 
\[
\pi^*(K_X+\Delta+M) = K_Y+\Delta_Y+M_Y.
\] 
Let $E$ be a divisorial valuation with log discrepancy in the interval $[m,m+\delta)$ and center in $W$.
Then, the center of $E$ on $Y$ is contained in $\pi^{-1}(W)$.
If $E$ is not a divisor on $Y$, then by condition $(3)$
its center has codimension two. Moreover, its center equals the intersection of a prime component of $\pi^{-1}(W)$ and a component of $\lfloor B_Y^{>0}\rfloor$.
There are finitely many such strata and finitely many such divisors over each stratum, so the statement follows.
\end{proof}

The following lemma allows us to compare the minimal log discrepancy along the diminished base locus in a minimal model program.

\begin{lemma}\label{lem:mld-diminished}
Let $(X,\Delta+M)$ be a projective generalized pair.
Assume that $K_X+\Delta+M$ is pseudo-effective. 
Let $\pi\colon (X,\Delta+M)\dashrightarrow (X_1,\Delta_1+M_1)$
be a step of a minimal model program for $(X,\Delta+M)$.
Let $B$ be the diminished base locus of $K_X+\Delta+M$
and let $B_1$ be the diminished base locus of $K_{X_1}+\Delta_1+M_1$.
Then, we have that 
\[
{\rm mld}(X,\Delta+M;B) \leq {\rm mld}(X_1,\Delta_1+M_1;B_1).
\] 
Furthermore, the inequality is strict if the exceptional locus of $\pi$ contains the center of every divisor computing the minimal log discrepancy of $(X,\Delta+M;B)$.
\end{lemma}

\begin{proof}
Let $E$ be a prime divisor over $X_1$ which computes the minimal log discrepancy of $(X_1,\Delta_1+M_1;B_1)$.
If the center of $E$ on $X_1$ is not contained in ${\rm Ex}(\pi^{-1})$, then we have that
\[
a_E(X,\Delta+M) = {\rm mld}(X_1,\Delta_1+M_1;B_1). 
\] 
Furthermore, we have that $c_X(E)\subset B$ by Proposition~\ref{prop:flips-vs-dim-base}.
We conclude that the inequality holds in this case. 
On the other hand, if $c_E(X_1) \subseteq {\rm Ex}(\pi^{-1})$, then we have that
$c_E(X_1)\subseteq {\rm Ex}(\pi) \subseteq B$.
The last contaiment follows since ${\rm Ex}(\pi)$ is covered by $(K_X+\Delta+M)$-negative curves.
By the monotonicity lemma~\ref{lem:gen-monotonicity}, we deduce that 
\[
a_E(X,\Delta+M)< {\rm mld}(X_1,\Delta_1+M_1;B_1).
\] 
Hence, the inequality always holds.

Now, assume that the exceptional locus of $\pi$ contains the center of every divisor computing the mld of $(X,\Delta+M;B)$.
Let $E$ be a prime divisor over $X_1$ computing the mld of $(X_1,\Delta_1+M_1;B_1)$.
If $c_{X_1}(E)$ is not contained in ${\rm Ex}(\pi^{-1})$, then we have 
\[
{\rm mld}(X,\Delta+M;B)< a_E(X,\Delta+M) = {\rm mld}(X_1,\Delta_1+M_1;B_1).
\]
The first inequality holds as $c_X(E)$ is not contained in ${\rm Ex}(\pi)$.
On the other hand, assume that $c_{X_1}(E)$ is contained in ${\rm Ex}(\pi^{-1})$. Then, we have that $c_X(E)\subseteq {\rm Ex}(\pi)$.
In this case, we have that
\[
{\rm mld}(X,\Delta+M;B)\leq a_E(X,\Delta+M) < {\rm mld}(X,\Delta_1+M_1;B_1).
\]
Thus, if ${\rm Ex}(\pi)$ contains the center of every divisor computing the mld of $(X,\Delta+M;B)$,
then the inequality is strict.
\end{proof}

\begin{remark}
{\em 
We observe that the conclusion of Lemma~\ref{lem:finiteness-e+d} does not hold if we take $\delta=1$. Indeed, consider the log pair $(\mathbb{A}^2)$. Its minimal log discrepancy equals $2$ and  is obtained by blowing-up the maximal ideal at the origin.
Let $X\rightarrow \mathbb{A}^2$ be such blow-up.
Any blow-up at a closed point of the exceptional divisor of $X\rightarrow \mathbb{A}^2$ induces a divisorial valuation with log discrepancy $3$. There are infinitely many of these.
}
\end{remark}

\begin{remark}
{\em 
We observe that the conclusion of  Lemma~\ref{lem:finiteness-e+d} does not hold if we drop the condition of $W$ being a Zarsiki closed subset. 
Indeed, if $X$ is a smooth surface and $W$ is a countable union of closed points, then ${\rm mld}(X;W)=2$.
However, any such closed point computes the minimal log discrepancy, so we have countable many divisorial valuations computing the minimal log discrepancy.
}
\end{remark}

The following lemma is motivated by~\cite[Proposition 2.3]{Amb99},
which is used by Ambro to prove the constructibility of the mld-stratification.

\begin{lemma}\label{lem:generic-vs-general}
Let $(X,\Delta+M)$ be a log canonical generalized pair.
Let $W\subset X$ be an irreducible subvariety. 
Let 
\[
m:={\rm mld}(X,\Delta+M;W).
\]
Let $\delta \in (0,1)$.
Assume there are at most $N$ divisors over $X$ with log discrepancy in the interval
$[m,m+\delta)$ and center $W$.
Then, there is an open set $U\subset X$, with $W\cap U\neq\emptyset$, satisfying the following.
For every closed point $x\in W\cap U$, there are at most $N$ divisors over $X$ with log discrepancy in the interval
\[
[{\rm mld}(X,\Delta+M;x),{\rm mld}(X,\Delta+M;x)+\delta) 
\] 
and center $x$ in $X$.
\end{lemma}

\begin{proof}
We denote by $r$ the codimension of $W$ in $X$.
Let $\pi\colon Y\rightarrow X$ be a log resolution of $(X,\Delta+M)$ on which $M$ descends.
We may assume that every divisor with center equal to $W$ and log discrepancy
in the interval $[m,m+\delta)$ appears in the model $Y$.
Furthermore, We may assume that $\pi^{-1}(W)$ is supported on the exceptional divisor of $\pi$.
We let $\cup_{i\in I}E_i$ to be the reduced exceptional divisor of $\pi$ 
plus the reduced strict transform of $B$.
Let $I_W\subset I$ be the subset for which $\pi(E_i)=W$. 
Then, by~\cite[Proposition 1.9]{Amb99} we have that 
\[
\mu^{-1}(W)= \cup_{i\in I_W}E_i 
\text{ and }
{\rm mld}(X,\Delta+M;W) = \min_{i\in I_W}a_{E_i}(X,\Delta+M).
\]
We write
\[
K_Y+\Delta_Y+M_Y=\pi^*(K_X+\Delta+M).
\]
Shrinking around the generic point of $W$, 
we may assume that every connected component of intersection of some of the $E_i$'s dominate $W$.
We denote by $U$ an open subset with $U\cap W\neq \emptyset$ 
for which the previous condition holds.
Let $x\in W \cap U$ be a closed point and $\eta \in \pi^{-1}(x)$. 
Then, we have that 
\[
{\rm mld}(Y,\Delta_Y+M_Y;\eta) = \sum_{E_i\supseteq \eta}(1-{\rm coeff}_{E_i}(\Delta_Y)) + (n-s),
\]
where $s$ is the cardinilaty of the set $\{ i\in I \mid E_i\supseteq \eta\}$.
We have that 
\[
n-s\geq r \text{ and } \sum_{E_i\supseteq \eta}(1-{\rm coeff}_{E_i}(\Delta_Y))\geq m.
\]
Hence, we conclude that 
\[
{\rm mld}(X,\Delta+M;x)\geq r+m.
\]
Let $E$ be a divisor on $Y$ which computes ${\rm mld}(X,\Delta+M;W)$.
By blowing-up a generic point of $\pi^{-1}(x)\cap E$, we see that
${\rm mld}(X,\Delta+M;x)=r+m$.

For every subset $J\subset I$ so that $\pi(\cap_{j\in J}(E_j))=W$, 
we have that 
\[
\sum_{j\in J}(1-{\rm coeff}_{E_j}(\Delta_Y)) > m+\delta. 
\] 
Hence, for $\eta\in \pi^{-1}(x)$, we have that 
\[
{\rm mld}(Y,B_Y+M_Y;\eta) \in [m+r,m+r+\delta)
\] 
if and only if $\eta$ is a codimension $r$ component 
of $\pi^{-1}(x)\cap E_i$. Here, $E_i$ is a divisor with center $W$
and log discrepancy in the interval $[m,m+\delta)$. 
By shrinking $U$ if necessary, we may assume that $E_i\cap \pi^{-1}(x)$ is irreducible
and of codimension $r$ in $Y$ for each such divisor $E_i$.
We deduce that for any closed $x\in U\cap W$ there are at most $N$ divisorial valuations
with log discrepancy in the interval $[m+r,m+r+\delta)$ and center $x$.
\end{proof}

The following lemma uses 
the classic fact that generalized klt singularities
can be approximated with klt singularities. 
We will use it to reduce problems about the regional fundamental
group of generalized klt singularities to klt singularities.

\begin{lemma}\label{lem:from-gen-to-pair}
Let $(X,\Delta+M)$ be a generalized pair.
Let $x\in X$ be a closed point. 
Assume that $(X,\Delta+M;x)$ is klt and
denote its minimal log discrepancy by $m$.
Fix $\delta \in (0,1)$.
Assume there are at most $N$ divisorial valuations over $X$
with center $x\in X$ and log discrepancy
in the interval $[m,m+\delta)$.
Then, for every $\delta'\in (0,\delta)$ 
there exists a boundary $B\geq \Delta$ so that 
$(X,B;x)$ is a klt pair and 
there are at most $N$ divisorial valuations over $X$
with center $x$ and log discrepancy
in the interval
\[
[{\rm mld}(X,B;x),{\rm mld}(X,B;x)+\delta')  
\] 
\end{lemma}

\begin{proof}
Let $\pi\colon Y\rightarrow X$ be a log resolution of the generalized pair
$(X,\Delta+M)$.
We know that $(X,\Delta+M)$ has at most $N$ divisorial valuations
with log discrepancy in the interval $[m,m+\delta)$.
We may assume these divisors are extracted on $Y$.
Write
\[
K_Y+\Delta_Y+M_Y = \pi^*(K_X+\Delta+M).
\]
The divisor $M_Y$ is nef and big over $X$.
Hence, by~\cite[Example 2.2.19]{Laz04a}, for every $k\geq 1$, we can write
\[
M_Y \sim_{\qq,X} A_k + \frac{E}{k},
\] 
where $A_k$ is a general effective ample divisor
and $E\geq 0$ is a fixed effective divisor independent of $k$.
For $k$ large enough, we have that 
\[
\left(Y,\Delta_Y + A_k +\frac{E}{k}\right) 
\] 
is a sub-klt pair. 
Let $B_k$ be the push-forward of $\Delta_Y+A_k+\frac{E}{k}$ to $X$.
Then, we have that $(X,B_k;x)$ is a klt pair.
Every prime divisor $F$ over $X$, with center $x$, which is exceptional over $Y$, satisfies that $a_F(X,\Delta+M)\geq m+\delta$.
By the continuity of log discrepancies with respect to the boundary, 
we can fix $k$ sufficiently large so that 
$a_F(X,B_k)>m+\delta'$ for every such $F$.
Thus, if $k$ is large enough, then there are at most $N$ divisorial valuations
over $X$ with center $x$ and log discrepancy in the interval 
\[
[{\rm mld}(X,B_k;x), {\rm mld}(X,B_k;x)+\delta').
\]
Hence, it suffices to take $B:=B_k$ for $k$ sufficiently large.
\end{proof}

The following proposition allows us to compare the log discrepancies
of a pair with those of a finite quotient.
For the definition of quotient generalized pair see~\cite[\S 2]{Mor21a}.

\begin{proposition}\label{prop:ld-finite-quot} 
Let $(X,\Delta+M;x)$ be a generalized klt singularity
and $G\leqslant {\rm Aut}(X,\Delta+M;x)$ a finite group.
Let $(Y,\Delta_Y+M_Y;y)$ be the quotient of $(X,\Delta+M;x)$ by $G$.
Let $\epsilon_0 < {\rm mld}(Y,\Delta_Y+M_Y;y)$.
Assume that $(Y,\Delta_Y+M_Y;y)$ has at most $N$ divisorial valuations
with center $y$ and log discrepancy in the interval $(\epsilon_0,\epsilon_1)$.
Then, $(X,\Delta+M;x)$ has at most $N|G|$ divisorial valuations with center $x$ and log discrepancy in the interval $(\epsilon_0,\epsilon_1)$.
\end{proposition}

\begin{proof}
Let $E$ be a prime divisor over $X$ with center $x$ and 
log discrepancy in $(\epsilon_0,\epsilon_1)$ with respect to $(X,\Delta+M)$.
Let $E^G$ be the $G$-closure of the corresponding divisorial valuation.
Then, in some $G$-equivariant birational model $Y$ of $X$, the $G$-closure $E^G$ is a divisor with at most $|G|$ prime components.
Let $F:=E^G/G$ be the quotient valuation on $Y$. 
By~\cite[Proposition 2.11]{Mor21a}, we have that 
\begin{equation}\label{eq:cont-quot} 
a_F(Y,\Delta_Y+M_Y) \in (\epsilon_0,\epsilon_1), 
\end{equation} 
and the center of $F$ on $Y$ equals $y$.
Hence, there are at most $|G|$ choices for $E$ given any of the
$N$ divisors over $Y$ satisfying containment~\eqref{eq:cont-quot}.
We conclude that there are at most $N|G|$ divisorial valuations with center $x$ and log discrepancy
in the interval $(\epsilon_0,\epsilon_1)$.
\end{proof}

\subsection{Regional fundamental groups}

In this subsection, we recall the definition
of the regional fundamental group of a klt singularity
as well as some theorems about this group.

\begin{definition}
{\em 
Let $(X,\Delta+M)$ be a generalized pair.
We define $\pi_1^{\rm reg}(X,\Delta+M)$,
called the {\em regional fundamental group} of $(X,\Delta+M)$,
to be the inverse limit of ${\rm Aut}_X(X_i)$,
with $i\in I$, 
where $I$ indexes the projective system of all the Galois quasi-\'etale covers $p_i\colon X_i\rightarrow X$ such that 
\[
p_i^*(K_X+\Delta+M)= K_{X_i}+\Delta_{X_i}+M_{X_i},
\]
where $\Delta_{X_i}\geq 0$.

Let $(X,B+M;x)$ be a generalized singularity.
We define the {\em regional fundamental group} of $(X,\Delta+M;x)$ denoted by 
$\pi_1^{\rm reg}(X,\Delta+M;x)$, 
to be the inverse limit of $\pi_1^{\rm reg}(U,\Delta_U+M_U)$, where the limit run over all the \'etale neighborhoods $U$ of $x$ in $X$.
Here, $\Delta_U$ and $M_U$ are defined by log pull-back.
}
\end{definition}

The following lemma will allow us to control
the denominators of log discrepancies once
we can control the regional fundamental group of a generalized singularity.

\begin{lemma}\label{lem:controlled-p1}
Let $(X,\Delta+M)$ be a generalized pair.
Let $x\in X$ be a closed point for which
\[
|\pi_1^{\rm reg}(X,\Delta+M;x)|\leq \rho.
\] 
Let $q\in \zz$ be so that
$q\Delta$ and $qM_Y$ are Weil divisors,
where $Y$ is a model in which $M$ descends.
Then, for any $E$ exceptional with $c_X(E)\supset x$, we have that
\[
q\rho a_E(X,\Delta+M) \in \zz. 
\]
\end{lemma}

\begin{proof}
By assumption, we have that 
$q(K_X+\Delta+M)$ is a Weil divisor.
By the bound on the regional fundamental group, we have that 
$q\rho(K_X+\Delta+M)$ is a Cartier divisor.
Indeed, the index one cover of $q(K_X+\Delta+M)$ has order at most $\rho$ around $x$.
Let $E$ be a prime divisor over $X$ with center through $x$.
Let $\pi\colon Y\rightarrow X$ be a log resolution of $(X,\Delta+M)$
on which $M$ descends and $c_Y(E)$ is a divisor.
Then, we can write
\[
q\rho(K_Y+B_Y+M_Y) = \pi^*(q\rho(K_X+\Delta+M)). 
\] 
Observe that $q\rho(K_Y+M_Y)$ is Weil
and $\pi^*(q\rho(K_X+\Delta+M))$ is Cartier.
We conclude that $q\rho B_Y$ is a Weil divisor.
In particular, 
\[
a_E(K_X+\Delta+M)=1-{\rm coeff}_E(B_Y) \in \zz\left[\frac{1}{q\rho}\right].
\]
This finishes the proof.
\end{proof}

\begin{definition}
{\em 
Let $(X,B+M)$ be a generalized pair.
The {\em model regularity} of $(X,B+M)$, denoted by
${\rm reg}_m(X,B+M)$ is the maximum number of $\qq$-Cartier prime components $S_1,\dots,S_r \subset \lfloor B\rfloor$
so that $S_1\cap \dots \cap S_r \neq \emptyset$.
If $\lfloor B\rfloor$ is empty, then we say that the model regularity is $-\infty$.
The {\em birational regularity} of a generalized pair $(X,B+M)$ is defined to be 
\[
{\rm reg}_b(X,B+M):=\sup \left\{ 
{\rm reg}_m(Y,B_Y+M_Y) \mid \text{ $(Y,B_Y+M_Y)$ is crepant to $(X,B+M)$}
\right\} -1. 
\] 
Here, crepant means that there exists a projective birational morphism $\pi\colon Y\rightarrow X$ for which
\[
K_Y+B_Y+M_Y = \pi^*(K_X+B+M).
\]
If the pair $(X,B+M)$ is log canonical,
then the regularity is in the set
$\{-\infty, 0,\dots, \dim X-1\}$. This follows from~\cite[Theorem 1]{MS21}. 
Let $(X,B+M;x)$ be a klt singularity.
We define its regularity to be 
\[
{\rm reg}(X,B+M;x):=\max \left\{ 
{\rm reg}_b(X,B'+M) \mid 
B'\geq B \text{ and $(X,B'+M)$ is log canonical} 
\right\}. 
\] 
The regularity of a klt singularity
is contained in the set $\{0,\dots, \dim X-1\}$.
If $(X,B+M;x)$ is a klt singularity of regularity $r$, we may say that it is a 
{\em $r$-regular} klt singularity.
}
\end{definition}

The following theorem is proved in~\cite[Theorem 7]{Mor21a}.

\begin{theorem}\label{thm:reg-fundamental} 
Let $n$ be a positive integer and $r\in [0,n-1]\cap \zz$.
There exists a constant $c(n)$, only depending on $n$,
satisfying the following.
Let $(X,B+M;x)$ be a $n$-dimensional $r$-regular gklt singularity.
Then, there exists a short exact sequence
\[
1\rightarrow A \rightarrow \pi_1^{\rm reg}(X,B+M)
\rightarrow N\rightarrow 1.
\] 
Here, $A$ is an abelian group of rank at most $r+1$
and $N$ is a group of order at most $c(n)$.
\end{theorem}

The following theorem allows comparing the regional fundamental groups of klt singularities in $X$ with those in the fibers of morphisms mapping to $X$ (cf.~\cite{Mor21a})

\begin{theorem}\label{thm:small-Q}
Let $n$ be a positive integer.
There exists a constant $c(n)$ which only depends on $n$ satisfying the following.
Let $(X,B+M;x)$ be a generalized klt singularity and
$\pi\colon Y\rightarrow X$ be a projective birational morphism.
There exists a boundary $B_Y$ supported on ${\rm Ex}(\pi)$ 
and a closed point $y\in Y$ for which 
\[
\pi_* \colon \pi_1^{\rm reg}(Y,B_Y+\pi^{-1}_*B+M_Y;y)
\rightarrow  \pi_1^{\rm reg}(X,B+M;x)
\] 
has cokernel of order at most $c(n)$.
In particular, if $\pi\colon Y\rightarrow X$ is a small morphism, then we have that the homomorphism
\[
\pi_* \colon \pi_1^{\rm reg}(Y,\pi^{-1}_*B+M_Y;y)
\rightarrow  \pi_1^{\rm reg}(X,B+M;x)
\] 
has cokernel of order at most $c(n)$.
\end{theorem}

The following lemma asserts that over a generalized klt singularity every finite abelian quasi-torsor is computed by index one covers of $\qq$-Cartier divisors (cf.~\cite[Proposition 4.10]{BM21}).
For the definition of abelian quasi-torsor, we refer the reader to~\cite[Definition 4.9]{BM21}. 
In this article, the only abelian quasi-torsors we will deal with are universal covers of singularities with abelian fundamental group.

\begin{lemma}\label{lem:index} 
Let $(X,B+M;x)$ be a generalized klt singularity.
Let $Y\rightarrow X$ be an abelian quasi-torsor.
Then, there exists $\qq$-Cartier divisors $D_1,\dots,D_k$ on $X$ through $x$, satisfying the following.
Let $K$ be the group that the $D_i$'s generate in ${\rm Cl}(X,x)$. 
Then, we have that 
\[
Y \simeq {\rm Spec}\left( \bigoplus_{D\in K} H^0(X,\mathcal{O}_X(D)) \right)
\] 
\end{lemma}

We finish this subsection by pointing out that 
the conjecture only depends on the formal equivalence class of the germ of the singularity (see, e.g.,~\cite[Definition 2.5]{Mor20c}).

\begin{proposition}\label{prop:germ-change}
Let $(X,B+M;x)$ and $(X',B'+M';x')$ be two generalized klt singularity which are formally equivalent.
Then, we have that 
\[
\pi_1^{\rm reg}(X,B+M;x) \simeq \pi_1^{\rm reg}(X',B'+M';x'),
\] 
\[
m={\rm mld}(X,B+M;x) = {\rm mld}(X',B'+M';x')
\] 
and the number of exceptional divisors over $X$ with center $x$ and log discrepancy in $[m,m+\delta)$
equals the number of exceptional divisors over $X'$ with center
$x'$ and log discrepancy in $[m,m+\delta)$.
\end{proposition} 

\begin{proof}
The first equality follows from the definition of the regional fundamental group.
The second equality and the last statement follow from the fact that the log discrepancies can be computed in the formal completion of $X$ at $x$.
\end{proof}

\subsection{Toric Singularities} 

In this subsection, we recall some basics about toric singularities. For the central results in toric geometry, we refer the reader to~\cite{CLS11,Ful93}.

\begin{definition}
{\em 
Let $N_\zz$ be a free finitely generated abelian group of rank $n$
Let $N_\qq$ be the associated $\qq$-vector space.
We denote by $M_\zz$ (resp. $M_\qq$) the dual of $N_\zz$ (resp. $N_\qq)$.
Let $\sigma \subset N_\qq$ be a strongly rational polyhedral cone. 
Assume that $\sigma$ is full-dimensional.
Then, the spectrum 
\[
X(\sigma):= 
{\rm Spec}(\kk[\sigma^\vee \cap M]) 
\] 
is called the {\rm affine toric variety} associated to $\sigma$.
The affine variety $X(\sigma)$ is endowed with a $n$-dimensional torus action ${\rm Spec}(\kk[M])$.
It has a unique fixed point for the torus action. 
Furthermore, any affine variety of dimension $n$
with an effective action of a $n$-dimensional torus
and a unique fixed point, is equivariantly isomorphic to $X(\sigma)$ for certain $\sigma\subset N_\qq$.
The prime torus invariant divisors of $X(\sigma)$ are in one-to-one correspondence with the rays of $\sigma$.
The set of rays of $\sigma$ is usually denoted by $\sigma(1)$.
The rank of the local Picard of $X(\sigma)$ at the fixed point $x$ equals the cardinality of $\sigma(1)$ minus $n$.
A divisor $B$ on $X(\sigma)$ through the fixed point, 
is said to be a {\rm torus invariant boundary} if it is supported on the union of the prime torus invariant divisors.
We say that $(X;x)$ is a {\em toric singularity}
if it is formally isomorphic to $(X(\sigma);x_0)$,
where $X(\sigma)$ is an affine toric variety and
$x_0$ is the unique fixed point.
}
\end{definition}

\begin{definition}
{\em 
An {\em affine generalized toric pair}
is a generalized pair $(X(\sigma),\Delta+M)$,
where $X(\sigma)$ is an affine toric variety
and $M$ is a b-nef toric divisor, i.e., 
there is a projective birational toric morphism $Y\rightarrow X(\sigma)$ where $M$ descends and equals a torus invariant divisor on $Y$.
A {\em generalized formally toric singularity} 
is a generalized singularity $(X,\Delta+M;x)$ which is formally isomorphic to an affine generalized toric pair at the unique torus invariant point.
}
\end{definition}

\begin{definition}
{\em
Let $\sigma \subset N_\qq$ be a rational polyhedral cone.
Let $(X(\sigma),\Delta;x)$ be a toric pair.
For each $\rho\in \sigma(1)$, we denote by $D_\rho$
the corresponding prime torus invariant divisor of $X(\sigma)$.
We denote by $n_\rho$ the largest positive integer for which
\[
1-\frac{1}{n_\rho} \leq {\rm coeff}_{D_\rho}(\Delta).
\] 
We call $\rho$ the {\em integral generator} of the ray $\qq_{\geq 0}\rho$.
The element $n_\rho^{-1}\rho \in N_\qq$ is called the {\em orbifold generator} of the ray $\qq_{\geq 0}\rho$.
The orbifold structure depends on the boundary $\Delta$.
We denote by $N_{\zz,\Delta}$, called the {\em orbifold lattice}, the lattice generated by $N$ and the $n_{\rho}^{-1}\rho$ for each $\rho \in \sigma(1)$.
}
\end{definition}

The following proposition is well known 
and it follows from the existence of toric log resolutions.

\begin{proposition}\label{prop:toric-mld}
Let $(X(\sigma),\Delta;x)$ be a toric singularity. 
For each $\rho \in \sigma(1)$, we denote by $D_\rho$ the corresponding prime torus invariant divisor on $X(\sigma)$.
Let $L_\Delta$ be the unique linear function on $N_\qq$
which takes the value $1-{\rm coeff}_{D_\rho}(\Delta)$ at $n_\rho^{-1}\rho$.
Then, we have that 
\[
{\rm mld}_x(X(\sigma),\Delta;x) = {\rm min}\{ 
L_\Delta(u) \mid u\in {\rm relint}(\sigma)^\vee \cap N_\zz 
\}.
\] 
\end{proposition}

\begin{proposition}\label{prop:reg-toric}
Let $(X(\sigma),\Delta;x)$ be a toric singularity.
Then, we have that 
\[
\pi_1^{\rm reg}(X(\sigma),\Delta;x) \simeq N_{\zz,\Delta}/ \langle
\rho  \mid \rho \in  \sigma(1) \rangle. 
\] 
\end{proposition}

\begin{proof}
Let $Z$ be the union of all the torus invariant strata
of codimension at least two.
Let $\Sigma$ be the fan of the toric variety $X(\sigma)\setminus Z$.
Let 
\[
\Delta(\Sigma) := \sum_{\rho\in \sigma(1)}\left( 
1-\frac{1}{n_\rho} D_\rho \right)|_{X(\Sigma)}
\]
Then, we have that 
\[
\pi_1^{\rm reg}(X(\sigma),\Delta;x) 
\simeq 
\pi_1(X(\Sigma), \Delta(\Sigma)).
\] 
The statement follows from~\cite[Theorem 12.1.10]{CLS11}.
\end{proof}

\begin{notation}
{\em 
Let $\sigma\subset N_\qq$ be a
strongly convex rational polyhedral cone. 
Let $m\in {\rm relint}(\sigma)$.
We denote by $\mathcal{S}(\sigma,m)$ the set of strongly convex rational polyhedral cones which are spanned by a proper subset of the rays of $\sigma$ and contain $m$ in their relative interior.
Note that $\mathcal{S}(\sigma,m)$ could be empty, for instance, if $\sigma$ is a simplicial cone and
$m$ is in its relative interior.}
\end{notation}

\begin{lemma}\label{lem:sub-cones}
Let $\sigma \subset N_\qq$ be a
full-dimensional strongly convex rational polyhedral cone.
Let $m\in {\rm relint}(\sigma)$.
Then, one of the following statements hold:
\begin{enumerate}
    \item $\sigma$ is a simplicial cone, or
    \item there exists $\tau\in \mathcal{S}(\sigma,m)$ which is full-dimensional in $N_\qq$, or
    \item there exists $\tau_1,\tau_2\in \mathcal{S}(\sigma,m)$ so that the span of 
    $\tau_1$ and $\tau_2$, over $\qq$, equals $N_\qq$.
\end{enumerate}
\end{lemma}

\begin{proof}
Let $n$ be the dimension of $N_\qq$ over $\qq$.
We assume that $\sigma$ is not simplicial.
Let $\rho \in \sigma(1)$.
We claim that there exists $\tau \in \mathcal{S}(\sigma,m)$
for which $\rho \in \tau$.
In particular, $\mathcal{S}(\sigma,m)$ is non-empty.
Indeed, for some $\lambda>0$, we have that
$m-\lambda \rho$ lies in the boundary of $\sigma$.
In particular, $m-\lambda \rho$ is in the relative interior
of a proper face $\tau_0$ of $\sigma$.
Note that $\rho$ is not in $\tau_0$, otherwise, we would have that $m\in \tau_0$, contradicting the fact that $m\in {\rm relint}(\sigma)$.
Let $\rho_1,\dots,\rho_k$ be the extremal rays of $\tau_0$.
First, assume that $\tau_0$ is a simplicial cone.
Then, $k$ must be at most $n-1$.
Since $\sigma$ has at least $n+1$ extremal rays, 
then the cone
\[
\tau:=\langle \rho, \rho_1,\dots,\rho_k\rangle_{\qq_{\geq 0}}
\] 
is contained in the set $\mathcal{S}(\sigma,m)$
and $\rho \in \tau$.
On the other hand, assume that $\tau_0$ is not a simplicial cone.
By induction on the dimension, we know that 
\[
\mathcal{S}(\tau_0,m-\lambda \rho) \neq \emptyset.
\]
Then, up to reordering, there exists $\{\rho_1,\dots,\rho_j\}\subsetneq \tau_0(1)$, with $j<k$, for which 
\[
m-\lambda \rho \in {\rm relint}(\langle \rho_1,\dots,\rho_j\rangle_{\qq_{\geq 0}}).
\] 
We conclude that $m\in {\rm relint}(\tau)$, 
where $\tau:=\langle \rho,\rho_1,\dots,\rho_j\rangle_{\qq_{\geq 0}}$.
Hence, the cone $\tau$ is contained in the set $\mathcal{S}(\sigma,m)$ and contains $\rho$. 
This shows the claim.

Assume that no $\tau\in \mathcal{S}(\sigma,m)$ is full-dimensional in $N_\qq$.
Let $\tau_1 \in \mathcal{S}(\sigma,m)$.
By assumption, $\tau_1$ is not full-dimensional.
Let $\rho\in \sigma(1)$ so that $\rho$ is not in the $\qq$-span of $\tau_1$.
By the claim, we can find $\tau_2 \in \mathcal{S}(\sigma,m)$ so that 
$\rho \in \tau_2$.
If the span of $\tau_1$ and $\tau_2$ over $\qq$ equals $N_\qq$, then we are done.
Otherwise, we replace $\tau_1$ with $\langle \tau_1,\tau_2\rangle_{\qq_{\geq 0}}$.
Observe that 
\[
\langle \tau_1, \tau_2\rangle_{\qq_{\geq 0}}\in
\mathcal{S}(\sigma,m).
\]
Furthermore, we have that 
\[
\dim (\langle \tau_1,\tau_2\rangle_{{\qq}_{\geq 0}}) > \dim(\tau_1).  
\]
Proceeding inductively, we produce $\tau_2$ as desired.  
Observe that this process must stop as the dimension of $\tau_1$ increases by at least one in each iteration.
\end{proof}

\section{Boundedness of the regional fundamental group} 

In this section, we restate the conjecture of the boundedness of the regional fundamental group in the setting of generalized pairs.
We prove the conjecture in the case of
generalized
toric singularities,
quotient singularities,
isolated 
$3$-fold singularities, and 
exceptional singularities.

\begin{conj}\label{conj:gen-boundedness}
Let $n$ and $N$
be two positive integers.
Let $0<\epsilon$ and $0<\delta<1$
be two real numbers.
There exists a constant
$\rho:=\rho(n,N,\epsilon,\delta)$,
only depending on the variables
$n,N,\epsilon$ and $\delta$,
which satisfies the following.
Let $(X,\Delta+M;x)$ be a $n$-dimensional
generalized klt singularity.
Assume the following statements hold:
\begin{enumerate}
    \item The inequality
    \[
    {\rm mld}(X,\Delta+M;x)>\epsilon
    \] 
    is satisfied, and 
    \item there are at most $N$ prime divisors $E$ over $X$ for which
    $c_E(X)=x$ and 
    \[
    a_E(X,\Delta+M;x) \in 
    [
    {\rm mld}(X,\Delta+M;x), 
    {\rm mld}(X,\Delta+M;x)+\delta 
    ).
    \]
\end{enumerate}
Then, the order of the regional fundamental group of $(X,\Delta+M;x)$ is bounded above by $\rho$, i.e., we have that
$|\pi_1^{\rm reg}(X,\Delta+M;x)|\leq \rho$.
\end{conj}

In order to prove the toric case, we will use the following theorem due to Lawrence (cf.~\cite{Law91}).

\begin{theorem}\label{thm:real-torus}
Let $T_\rr=\rr^d/\zz^d$ be the $d$-dimensional real torus.
\begin{enumerate}
    \item Let $U\subset T_\rr$ be an open set.
    Then, the set of closed subgroups of $T_\rr$ not intersecting $U$ has only finitely many maximal elements with respect to the inclusion.
    \item The set of finite unions of closed subgroups of $T_\rr$ satisfies the descending chain condition.
\end{enumerate}
\end{theorem}

Base on Lawrence's theorem, we will prove the following lemma regarding sequences of closed subgroups of the real torus.

\begin{lemma}\label{lem:closed-subgroup}
Let $\{S_i\}_{i\geq 1}$, with $i\geq 1$, be a sequence
of closed subgroups in $T_\rr$.
Then, up to passing to a subsequence, we can find a closed subgroup $S$ of $T_\rr$ so that 
for every $k\geq 1$, we have that
\[
S = \overline{\cup_{i\geq k} S_i}.
\]
\end{lemma}

\begin{proof}
Let $U=T_\rr \setminus \overline{  \cup_{i\geq 1} S_i}$.
By Theorem~\ref{thm:real-torus} (i), there are finitely many maximal closed subgroups which are disjoint from $U$.
Let $S^0$ be one of such maximal closed subgroups
which contains infinitely many of the $S_i$'s.
Up to passing to a subsequence, we may assume that $S_i\subset S^0$
for all $i$ and
$S^0\cap U=\emptyset$.
Hence, for every $k\geq 1$, we have that 
\[
S^0\supseteq \overline{\cup_{i\geq k} S_i}. 
\]
Assume that there exists some $k$ for which the inclusion is strict.
Then, we can find $s\in S^0$ and an open neighborhood $U_s$ of $s$ in $T_\rr$ so that
$U_s\cap S_i=\emptyset$ for all $i\geq k$.
By Theorem~\ref{thm:real-torus} (i), we can find a maximal closed subgroup $S^1$, disjoint from $T_\rr \setminus S \cup U_s$, which contain infinitely many of the $S_i$'s. 
Passing to a subsequence, we may assume that $S^1$ contains all the $S_i$'s. 
Observe that $S^0 \supsetneq S^1$.
If the inclusion 
\[
S^1 \supset \overline{\cup_{i\geq k} S_i}
\] 
is strict for some $k$, then we proceed as above and construct a closed subgroup $S^2$ properly contained in $S^1$. 
Inductively, we obtain a descending sequence of closed subgroups
\[
S^0\supsetneq S^1 \supsetneq S^2 \supsetneq \cdots .
\]
By Theorem~\ref{thm:real-torus} (ii), we conclude that this sequence must stabilize.
Hence, for some closed subgroup $S^j$, and up to passing to a  subsequence of the $S_i$'s, we
have that
\[
S^j = \overline{\cup_{i\geq k} S_i},
\]
for every $k\geq 1$.
\end{proof}

\begin{proposition}\label{prop:toric-case}
Conjecture~\ref{conj:gen-boundedness} holds for generalized formally toric singularities.
\end{proposition}

\begin{proof}
We prove the statement in four steps.\\

\textit{Step 1:} In this step, we reduce to the case of affine toric varieties.\\

Let $(T,\Delta_T+M_T;t)$ be a $n$-dimensional
generalized formally toric singularity
satisfying the hypotheses of the conjecture.
We may assume that $t$ is a torus invariant point.
By Lemma~\ref{lem:from-gen-to-pair}, 
it suffices to prove the statement for
formally toric singularities.
Hence, without loss of generality, 
we may assume that $M$ is the trivial b-divisor.
By Proposition~\ref{prop:germ-change}, 
we may assume that $(T,\Delta_T;t)$ is an affine toric variety.\\

\textit{Step 2:} In this step, we prove the case of $\qq$-factorial affine toric varieties.\\

We prove the case in which $(T,\Delta_T;t)$ is a $\qq$-factorial
affine toric variety. 
Let $\sigma \subset N_\qq$ be the full-dimensional simplicial rational polyhedral cone 
for which 
$T\simeq X(\sigma)$. 
Let $L_{\Delta}$ be the linear function on $N_\qq$ which corresponds 
to the log discrepancy
with respect to $K_T+\Delta_T$.
We denote by
$\rho_1,\dots,\rho_n\in N_\qq$ 
the lattice generators of the rays of $\sigma$.
We consider the $\qq$-linear invertible
transformation
$\Psi\colon N_\rr \rightarrow \rr^n$
which sends the lattice generators of
the rays of $\sigma$
to a canonical base of $\rr^n$.
By abuse of notation, 
we may identify $L_{\Delta}$
with $L_{\Delta} \circ \Psi^{-1}$.
We consider the quotient
\[
Q\colon \rr^n \rightarrow \rr^n/\zz^n\simeq T_\rr.
\]
Note that the function $L_\Delta$ is a well-defined linear function on 
$(0,1)^n \subset T$.
We denote by 
$S_\sigma$ the finite closed subgroup
\[
Q(\Psi(N_{\zz,\Delta_T})). 
\]
By Proposition~\ref{prop:reg-toric}, we know that
\[
|\pi_1^{\rm reg}(T,\Delta_T;t)| =
|\pi_1^{\rm reg}(T,\Delta_T^{\rm st};t)|= 
|S_\sigma|. 
\] 
Hence, it suffices to show that the order of the group $S_\sigma$ is bounded above whenever $(X(\sigma),\Delta_T;t)$ satisfies the conditions of the conjecture.

By contradiction, assume that there exists a sequence of cones $\sigma_i$
and pairs $(X(\sigma_i),\Delta_i;x)$,
which satisfy the hypotheses of the conjecture, 
so that $|S_{\sigma_i}|$ diverges.
Note that each pair
$(X(\sigma_i),\Delta_i;x)$ 
induces a log discrepancy 
function $L_i$ on $\rr^n$ as explained in the previous paragraph.
These log discrepancy functions
coverge to a linear function $L_\infty$.
By Lemma~\ref{lem:closed-subgroup}, up to passing to a subsequence, we can find a closed subgroup $S\leqslant T_\rr$ so that
\begin{equation}\label{eq:conv} 
S = \overline{\cup_{i\geq k} S_{\sigma_i}} 
\end{equation} 
holds for every $k\geq 1$.
We may assume that $S\cap (0,1)^n \neq \emptyset$.
Otherwise, the statement follows from induction on the dimension.
By construction, we have that
the infimum of $L_\infty$ in $S\cap (0,1)^n$ is at least $\epsilon$.
We let $m_\infty$ be the infimum of $L_\infty$ 
in
$S\cap (0,1)^n$.
Let $s_\infty \in S\cap (0,1)^n$ be so 
that $L_\infty(s_\infty) \in [m_\infty, m_\infty+\delta)$.
We choose $\epsilon_\infty>0$ small enough so that
$B(s_\infty,\epsilon_\infty) \subset (0,1)^n$.
By construction, we have that $\lim_{i\geq 1}m_i=m_\infty$,
where $m_i$ is the minimum of $L_i$
in $S_{\sigma_i}\cap (0,1)^n$.
By equality~\eqref{eq:conv},
we can choose $i$ large enough 
and $s_i\in S_{\sigma_i}$ so that 
$||s_i-s_\infty|| < \frac{\epsilon_\infty}{2}$.
Let $M$ be a positive integer.
Since $|S_{\sigma_i}|\rightarrow \infty$, we can find
$s'_i \in S_{\sigma_i}$ so that some representative
of $s'_i$ in $\rr^n$ has norm at most $\frac{\epsilon_\infty}{2M}$.
Hence, for $i$ large enough, we have that
\[
L_i(s_i + js'_i) \in [m_i,m_i+\delta) 
\] 
for every $j\in \{1,\dots, M\}$. 
Thus, for $i$ large enough, the toric pair $(X(\sigma_i),\Delta_i;x_i)$
has more than $M$ divisorial valuations with center
$x_i$ and log discrepancy in the interval $[m_i,m_i+\delta)$.
This leads to a contradiction. 
We conclude that $|S_{\sigma_i}|$ converges, so it must be bounded above.
Thus, $|S_{\sigma_i}|$ is bounded above by a constant which only depends on $n,N,\epsilon,\delta$. \\

\textit{Step 3:} In this step, we prove the existence of a special blow-up for affine toric varieties.\\

First, we will prove the following claim.\\

\textbf{Claim:} There exists a constant $k:=k(n,N,\epsilon,\delta)$, satisfying the following.
Let $(X(\sigma),\Delta;t)$ be a $n$-dimensional affine toric variety satisfying the hypothesis of the conjecture.
Let $m \in {\rm relint}(\sigma)\cap N$ be a lattice point computing the minimal log discrepancy. 
For $i\in \{1,\dots,n\}$, we can find $v_1,\dots,v_n$, linearly independent, satisfying the following: 
\begin{enumerate}
    \item For each $i$, we can write 
    $v_i=\sum_{\rho\in \sigma(1)}k_{i,\rho}\rho$ for certain non-negative integers $k_{i,\rho}$, 
    \item we have that $k_0m=\sum_{i=1}^n v_i$ for some positive integer $k_0$, and 
    \item there is an upper bound
    \[
    k_0 +\sum_{i=1}^n \sum_{\rho\in \sigma(1)}k_{i,\rho}\leq k.
    \]
\end{enumerate}

\begin{proof}[Proof of the Claim]
We will proceed by lexicographic induction on $(n,|\sigma(1)|)$.
If $\sigma$ is a simplicial cone, then the statement follows from the second step. 
If $\mathcal{S}(\sigma,m)$ contains a full-dimensional element, the statement follows by induction on $|\sigma(1)|$.
By Lemma~\ref{lem:sub-cones}, we may assume that $\mathcal{S}(\sigma,m)$ contains two elements
$\tau_1$ and $\tau_2$ so that they generate $N_\qq$. 
Let $d_i$ be the dimension of $\tau_i$.
Observe that $m$ is the minimizer for the log discrepancy function
in both  ${\rm relint}(\tau_1)\cap N$ and ${\rm relint}(\tau_2)\cap N$.
Hence, we may assume that the claim holds for both $\tau_1$ and $\tau_2$ by induction on the dimension. 
For each $i\in \{1,2\}$, we can find  
$v^i_1,\dots,v^i_{d_i}$ linearly independent
satisfying the following:
\begin{enumerate}
    \item For each $j$, we can write 
    $v^i_j=\sum_{\rho\in \tau_i(1)} k^i_{j,\rho} \rho$, for certain non-negative integers $k^i_{j,\rho}$, 
    \item we have that $k_0^i m = \sum_{j=1}^{d_i} v^i_j$ for some positive integer $k_0^i$, and
    \item there is an upper bound 
    \[
    k_0^i + \sum_{i=1}^2 \sum_{j=1}^{d_i} \sum_{\rho\in \tau_i(1)} k^i_{j,\rho} \leq  
    k(d_i,N,\epsilon,\delta).
    \] 
\end{enumerate}
Up to re-ordering, we may assume that the elements of the set
$\mathcal{B}:=\{v^1_1,\dots,v^1_{d_1},v^2_1,\dots,v^2_{n-d_1}\}$ are linearly independent and generate $N_\qq$.
Let $w:=v^2_{n-d_1+1}+\dots+v^2_{d_2}$.
For some $v\in \mathcal{B}$, we have that 
the elements of the set
$\mathcal{B}':=\mathcal{B}\setminus \{v\} \cup \{w+v\}$ are linearly independent. 
Let $v_1,\dots,v_n$ be the elements of the set $\mathcal{B}'$.
Since $\tau_i(1)\subset \sigma(1)$, the statement follows by taking 
\[
k(n,N,\epsilon,\delta) = 2\max\left\{ k(d,N,\epsilon,\delta) \mid d\in\{1,\dots,n-1\}\right\}.
\]
\end{proof}

Let $(X(\sigma),\Delta;x)$ be a
$n$-dimensional affine toric variety
satisfying the hypotheses of the conjecture.
Let $v_1,\dots,v_n$ be the lattice elements constructed in the claim.
For each $v_i$, we let $v_{i,0}$ be the lattice generator in $N_{\zz,\Delta}$ of the the ray spanned by $v_i$.
We denote by $\sigma_0$ the cone generated by the $v_{i,0}$'s.
By construction, we have that $\sigma_0 \subset \sigma$. 
Let $\phi\colon Y\rightarrow X$ be a projective birational toric blow-up which corresponds to a refinement of $\sigma$ containing $\sigma_0$ as a cone.
For each $i$, we have that 
\[
L_\Delta(v_{i,0}) = n_i L_\Delta(v_i) = n_i k_i = k_{i,0} \leq k(n,N,\epsilon,\delta) = k.
\]
For each $i$, we denote by $E_i$ the torus invariant divisor on $Y$ corresponding to the ray $v_{i,0}$. We let $y$ be the closed point in $Y$ corresponding to the interior of the cone $\sigma_0$.
Then, we can write 
\[
\phi^*(K_X+\Delta)=K_Y+\sum_{i=1}^n (1-k_{i,0})E_i + \Delta_Y,
\] 
where $\Delta_Y$ is a torus invariant divisor which does not intersect $y$.
By construction, we have that 
\[
m={\rm mld}(X_\sigma,\Delta;x) = 
{\rm mld}\left(Y,\sum_{i=1}^n (1-k_{i,0})E_i;y \right). 
\] 
The toric sub-pair
$(Y,\sum_{i=1}^n(1-k_{i,0})E_i;y)$ 
has at most $N$ divisorial valuations with center $y$ and log discrepancy in the interval $[m,m+\delta)$.
Furthermore, we have that 
\[
N_{\zz,\Delta} /\langle v_1,\dots, v_n \rangle_\zz \rightarrow 
\pi_1^{\rm reg}(X(\sigma),\Delta;x) 
\] 
is a surjective homomorphism.\\

\textit{Step 4:} We give an upper bound for
$N_{\zz,\Delta} /\langle v_1,\dots,v_n\rangle_\zz$ which only depends on $n,N,\epsilon$ and $\delta$.\\

We consider the toric pair 
\[
\left(Y, \sum_{i=1}^n \left(1-\frac{k_{i,0}}{k}\right)E_i;y\right). 
\] 
Then, we have that 
\[
{\rm mld} \left(Y, \sum_{i=1}^n \left(1-\frac{k_{i,0}}{k}\right)E_i;y\right) = \frac{m}{k}.
\] 
Furthermore, 
there are at most $N$ divisorial valuations over $Y$ with center $y$ and log discrepancy in the interval $\left[\frac{m}{k},\frac{m}{k}+\frac{\delta}{k}\right)$. By the second step, we conclude that 
\begin{equation}\label{eq:controlled-blow-up}
\left| 
\pi_1^{\rm reg}\left(Y,\sum_{i=1}^n\left(1-\frac{k_{i,0}}{k}\right) E_i ; y \right) 
\right| \leqslant \rho\left( 
n,N,\frac{\epsilon}{k},\frac{\delta}{k}
\right).
\end{equation} 
In particular, 
we have that 
\[
1-\frac{k_{i,0}}{k} \leq 1-\frac{1}{\rho\left(
n,N,\frac{\epsilon}{k},\frac{\delta}{k}\right)}.
\]
Hence, we conclude that $n_i \leq \rho(n,N,\epsilon k^{-1},\delta k^{-1})$.
By inequality~\eqref{eq:controlled-blow-up}, we conclude that 
\[
\left| 
N_{\zz,\Delta} /\langle v_{1,0},\dots,v_{n,0}\rangle_\zz 
\right| \leq \rho\left(
n,N,\frac{\epsilon}{k},\frac{\delta}{k}\right).
\]
We conclude that 
\begin{equation}\label{eq:finaleq}
\left| 
N_{\zz,\Delta} /\langle v_{1},\dots,v_{n}\rangle_\zz 
\right| \leq \rho\left(
n,N,\frac{\epsilon}{k},\frac{\delta}{k}\right)^{n+1}.
\end{equation} 
Since $k$ only depends on $n,N,\epsilon$ and $\delta$, we conclude that the value on the righ side hand of inequality~\eqref{eq:finaleq}, only depends on $n,N,\epsilon$, and $\delta$.
\end{proof}

\begin{proposition}\label{prop:quot-case}
Conjecture~\ref{conj:gen-boundedness} holds
for formally quotient singularities.
\end{proposition}

\begin{proof}
By Proposition~\ref{prop:germ-change}, 
we may assume that $(X,\Delta;x)$ is an affine $n$-dimensional log quotient singularity.
Hence, we can write
\[
(X,\Delta;x) \simeq (\mathbb{A}^n;\{0\})/G,
\]
where $G$ is a finite group fixing then origin.
Let $E$ be the divisor over $X$ computing
the minimal log discrepancy
of $(X,\Delta;x)$.
Let $F$ be a $G$-prime divisorial valuation over $\mathbb{A}^n$ so that
its quotient is the divisorial valuation
corresponding to $E$.
We can find a $G$-equivariant projective birational morphism $Y\rightarrow \mathbb{A}^n$ so that $F$ is a $G$-invariant $G$-prime divisor 
on $Y$.
Let $G_F$ be the largest normal subgroup
which acts as the identity on $F$.
Then, $G_F$ is a cyclic subgroup of $G$
(see, e.g.,~\cite[Lemma 2.6.(b)]{FZ05}).
Let
\[
(T,\Delta_T;t):=(\mathbb{A}^n;\{0\})/G_F.
\]
Since the action of $G_F$ on $\mathbb{A}^n$ is diagonalizable, the
pair $(T,\Delta_T)$ is toric.
From~\cite[2.42.4]{Kol13}, we conclude that 
\[
{\rm mld}(T,\Delta_T;t):=
{\rm mld}(X,\Delta;x)
\] 
Let $H:=G/G_T$ be the quotient group
acting on $(T,\Delta_T;t)$.
Then, we have that 
\[
(X,\Delta;x) \simeq (T,\Delta_T;t)/H.
\]
By~\cite[Theorem 2]{Mor20c}, there exists a normal abelian subgroup $A\leqslant H$ so that
$A< \mathbb{G}_m^r \leqslant {\rm Aut}(T,\Delta_T;t)$ 
and the index of $A$ in $H$ is bounded
by a constant $c(n)$, 
which only depends on the dimension $n$ of $T$.
Let $(T_0,\Delta_{T_0};t_0)$ be the log quotient
of $(T,\Delta_T;t)$ by the abelian group $A$.
Then, $(T_0,\Delta_{T_0};t_0)$ is a toric singularity.
On the other hand, we have that
\begin{equation}\label{eq:mld>e} 
{\rm mld}(T_0,\Delta_{T_0};t_0) =
{\rm mld}(X,\Delta;x) > \epsilon.
\end{equation} 
Denote by $N_0:=H/A$.
Then, $N_0$ is a group of order at most $c(n)$
which acts on $(T_0,\Delta_{T_0};t_0)$ 
so that
\[
(X,\Delta;x)\simeq (T_0,\Delta_{T_0};t_0)/N_0.
\] 
By assumption, 
there are at most 
$N$ divisors over $X$, with center $x$ and
log discrepancy in the interval
\begin{equation}\label{eq:mld-int-1}
[{\rm mld}(X,\Delta;x), {\rm mld}(X,\Delta;x)+\delta). 
\end{equation} 
By Proposition~\ref{prop:ld-finite-quot},
we conclude that the toric pair
$(T_0,\Delta_{T_0};t_0)$ admits at most
$Nc(n)$ divisorial valuations with log discrepancy
in
\[
[{\rm mld}(T_0,\Delta_{T_0};t_0), 
{\rm mld}(T_0,\Delta_{T_0};t_0)+\delta)
\]
and center $t_0$.
By Proposition~\ref{prop:toric-case}, 
we conclude that there exists a constant
$\rho:=\rho(n,c(n)N,\epsilon,\delta)$ for which 
\[
|\pi_1^{\rm reg}(T_0,\Delta_{T_0};t_0)| \leq \rho.
\]
Note that the constant $\rho$ only
depends on $n,N,\epsilon$ and $\delta$.
On the other hand, we have a homomorphism 
\[
\pi_1^{\rm reg}(T_0,\Delta_{T_0};t_0)\rightarrow
\pi_1^{\rm reg}(X,\Delta;x) 
\] 
with cokernel of order at most $c(n)$.
We conclude that 
\[
|\pi_1^{\rm reg}(X,\Delta;x)| \leq 
\rho + c(n),
\]
which is a constant that
only depends on $n,N,\epsilon$ and $\delta$.
\end{proof}

\begin{proposition}\label{prop:3-fold-case}
Conjecture~\ref{conj:gen-boundedness} holds for generalized formally isolated $3$-fold singularities. 
\end{proposition} 

\begin{proof}
Let $(X,M;x)$ be a generalized formally isolated $3$-fold singularity.
By Proposition~\ref{prop:germ-change}, 
we may assume that $(X,M;x)$ is an affine $3$-dimensional isolated generalized klt singularity.
By Lemma~\ref{lem:from-gen-to-pair}, 
we may assume that $(X,\Delta;x)$ is a log pair. 
Furthermore, we can assume that all the coefficients of the boundary $\Delta$ are less than $\frac{1}{2}$.
Then, we have that
\[
\pi_1^{\rm reg}(X,\Delta;x) \simeq 
\pi_1^{\rm reg}(X;M;x). 
\] 
Let $\pi\colon Y \rightarrow X$ be a small $\qq$-factorialization of $(X,\Delta)$. 
By~\cite[Corollary 3]{Mor21a}, there exists a closed point $y\in Y$ for which 
\[
\pi_*\colon 
\pi_1^{\rm reg}(Y;y)\rightarrow \pi_1^{\rm reg}(X,\Delta)
\] 
is almost surjective with respect to the dimension, i.e., 
the cokernel of the homomorphism is bounded by a constant
$c(n)$, which only depends on $n=\dim X$.
Hence, it suffices to prove that the regional fundamental group of $Y$ at $y$ is bounded above
by a constant only depending 
on $n,N,\epsilon$ and $\delta$.
By Theorem~\ref{thm:reg-fundamental}, 
where exists an exact sequence
\[
1\rightarrow 
A_0\rightarrow 
\pi_1^{\rm reg}(Y,y)\rightarrow N_0 \rightarrow 1,
\]
where $A_0$ is a free finite abelian group
of rank at most $3$
and $N_0$ is a finite group of order at most $c'(n)$.
Hence, it suffices to prove that the 
order of each generator of $A_0$ is bounded
above by a constant which only depends
on $n,N,\epsilon$ and $\delta$.
Let 
\[
(Y';y') \rightarrow (Y;y) 
\]
be the quasi-\'etale Galois cover corresponding to the quotient group $N_0$.
Then, we have that 
\[
\pi_1^{\rm reg}(Y';y') \simeq A \simeq \zz_{k_1}\oplus \zz_{k_2} \oplus \zz_{k_3}. 
\] 
We claim that, for $\epsilon_0$ small enough, there are at most $c'(3)N$ exceptional divisorial valuations over $Y'$ whose centers map through $y'$ and
that have log discrepancy in the interval $(\epsilon_0,1+\epsilon_0)$.
Indeed, if we choose $\epsilon_0 <\epsilon$ small enough so that
$1+\epsilon_0 < \epsilon+\delta$, then the germ $(X,x)$ has at most $N$ divisorial valuations with center $x$ 
with log discrepancy in the interval $(\epsilon_0,1+\epsilon_0)$.
The same statement holds for $(Y,y)$.
Finally, since the degree of the finite map
$Y'\rightarrow Y$ is bounded by $N_0$,
we conclude the claim by Proposition~\ref{prop:ld-finite-quot}. 
By~\cite[Theorem 3.18]{BM21}, we know that $Y$ is locally a Mori dream space around $y$.
In particular, by Lemma~\ref{lem:index}, the universal cover of $(Y',y')$ is realized by the index one cover of Weil $\qq$-Cartier divisors.
In particular, the positive integers $k_1,k_2$, and $k_3$ are bounded above by the Cartier index of a Weil $\qq$-Cartier divisor on $Y'$ through $y'$.
By~\cite[Lemma 4.4.1]{Sho96}, the Cartier index of any Weil $\qq$-Cartier divisor on $Y'$ through $y'$ is bounded above by a constant $\rho_0(c'(3)N,\epsilon,\delta)$, 
which only depends on $N,\epsilon$ and $\delta$.
We conclude that 
\[
|\pi^{\rm reg}(X,M;x)| \leq
\rho(N,\epsilon,\delta):= 
c(3)+c'(3)+\rho_0(c'(3)N,\epsilon,\delta)^3. 
\] 
\end{proof}

\begin{proposition}\label{prop:ex-case}
Conjecture~\ref{conj:gen-boundedness} holds for generalized klt formally exceptional singularities.
\end{proposition}

\begin{proof}
By Proposition~\ref{prop:germ-change}, we may assume that we have an affine generalized klt exceptional singularity. 
Let $(X,\Delta+M;x)$ be a generalized klt exceptional singularity.
By Theorem~\ref{thm:reg-fundamental}, the regional fundamental group
$\pi_1^{\rm reg}(X,\Delta+M;x)$ is the extension of a finite cyclic group
with a finite group 
whose order is bounded above by a constant $c(n)$
only depending on the dimension $n$.
Let $(X',\Delta'+M';x)$ be its universal cover. 
By~\cite[Theorem 1]{Mor20}, we can find a projective birational morphism
$\pi\colon Y\rightarrow X$ that extracts the unique exceptional divisor $E$ over $x$ which can be a log canonical center through $x$. 
Let $B\geq \Delta$ be an effective divisor
so that $(X,B+M;x)$ is strictly log canonical at $x\in X$. 
Then, we obtain a commutative diagram as follows:
\[
\xymatrix{
Y'\ar[d]^-{\pi'}\ar[r] & Y \ar[d]^-{\pi}  \\
(X',B'+M';x')\ar[r]^-{\phi} & (X,B+M;x).
}
\]
Here, the vertical arrows are projective birational morphisms
and the horizontal arrows are finite quotients.
By the connectedness theorem, $Y'\rightarrow X'$ extracts a unique prime divisor $E'$ which is a
generalized log canonical center of $(X',B'+M';x')$.
Since $E$ is an exceptional Fano type variety, then it belongs to a bounded family (see, e.g.,~\cite[Theorem 1.2]{Bir21}).
Hence, $E'$ also belongs to a bounded family.
In particular, we can find a curve $C\subset E'$ which lies in the smooth locus and so that
$-K_{E'} \cdot C \leq k(n)$, for a constant $k(n)$, 
which only depends on the dimension $n$ (see, e.g.,~\cite[Lemma 2.19]{Mor18b}).
In particular, we have that the intersection of
$C$ with the strict transform of $B'$ on $Y'$ is bounded above by $k(n)$.
Write 
\[
K_{Y'}+B_{Y'}+E'+M_{Y'} = {\pi'}^*(K_{X'}+B'+M') 
\text{ and }
K_{Y'}+\Delta_{Y'}+\alpha E' +M_{Y'}=
{\pi'}^*(K_{X'}+\Delta'+M'). 
\] 
Hence, we conclude that 
\[
(1-\alpha)E' + (B_{Y'}-\Delta_{Y'}) \equiv_{X'} 0.
\] 
Observe that $-E'\cdot C$ is bounded below 
and $(B_{Y'}-\Delta_{Y'})\cdot C$ is bounded above by
a constant which only depends on $n$.
We conclude that $\alpha$ is bounded below by a constant that only depends on $n$.
This means that the log discrepancy at $E'$ 
for $(X',\Delta'+M';x')$ is bounded above
by a constant $A(n)$, which only depends on the dimension $n$.
Since ${\rm mld}(X,\Delta+M;x)>\epsilon$, we conclude that the ramification index of $Y'\rightarrow Y$ at $E'$ must be bounded above by a constant which only depends on the dimension. 
By~\cite[Theorem 1]{Mor21a}, we know that such ramification index is the order of the finite cyclic group appearing in the extension of $\pi_1^{\rm reg}(X,\Delta+M;x)$.
Hence, the regional fundamental group of $\pi_1^{\rm reg}(X,\Delta+M;x)$ is bounded above by a constant 
only depending on $\epsilon$ and $n$. 
\end{proof} 

\begin{proof}[Proof of Theorem~\ref{introthm:boundedness}]
The theorem follows from Proposition~\ref{prop:toric-case},
Proposition~\ref{prop:quot-case},
Proposition~\ref{prop:3-fold-case}, and 
Proposition~\ref{prop:ex-case}.
\end{proof} 

\section{Termination of flips with scaling}

In this section,
we discuss the termination of flips with scaling under the assumption that the boundedness conjecture on the regional fundamental group holds.
First, we introduce the following conjecture which is inspired by the classic Conjecture~\ref{conj:dim-base} for pairs.

\begin{conj}\label{conj:dim-base-gen} 
Let $(X,\Delta+M)$ be a generalized klt pair.
Then, the diminished base locus $B_{\rm -}(K_X+\Delta+M)$ is Zariski closed in $X$. 
\end{conj}

Conjecture~\ref{conj:dim-base-gen} is known up to dimension $4$.
This result follows from the existence of minimal models for generalized pseudo-effective pairs of dimension at most $4$~\cite[Theorem 1]{Mor18a}. 
The diminished base locus of the divisor $K_X+\Delta+M$ is exactly the locus which is contracted or flipped in a minimal model program with scaling for $K_X+\Delta+M$ (see Lemma~\ref{lem:MMP-vs-dim}).
Thus, Conjecture~\ref{conj:dim-base-gen} is actually necessary for the termination of flips with scaling, as the exceptional locus of a minimal model program is a Zariski closed subset. The following is the main theorem of this subsection.

\begin{theorem}\label{thm:term-scaling}
Assume that the following statements hold:
\begin{enumerate} 
\item The boundedness of the regional fundamental group, Conjecture~\ref{conj:gen-boundedness}, holds in dimension $n$,
\item  the upper bound for minimal log discrepancies, Conjecture~\ref{conj:upp-mld}, holds in dimension $n$, and
\item the Zariski closedness of the diminished base locus, Conjecture~\ref{conj:dim-base-gen}, holds in dimension $n$.
\end{enumerate} 
Then, termination of flips with scaling holds in dimension $n$.
\end{theorem}

\begin{proof} 
First, we show that Conjecture~\ref{conj:upp-mld} implies the existence of an upper bound for the minimal log discrepancy of a generalized klt pair of dimension $n$.
Let $(X,\Delta+M;x)$ be a generalized klt singularity of dimension $n$.
Let $\phi\colon Y\rightarrow X$ be a small $\qq$-factorialization of $X$.
Write $\Delta_Y$ for the strict transform of $\Delta$ on $Y$ and
$M_Y$ the trace of the b-divisor on $Y$. 
Then, we have that
\[
{\rm mld}(X,\Delta+M;x) 
= 
{\rm mld}(Y,\Delta_Y+M_Y;\phi^{-1}(x)). 
\] 
On the other hand, we have that
\[
{\rm mld}(Y,\Delta_Y+M_Y;\phi^{-1}(x)) 
\leq 
{\rm mld}(Y;\phi^{-1}(x)) 
\leq A(n),
\] 
where the last inequality follows from Conjecture~\ref{conj:upp-mld}. Here, $A(n)$ is a constant which only depends on the dimension $n$.
Hence, we conclude that
\[
{\rm mld}(X,\Delta+M;x) 
\leq A(n),
\] 
as claimed. Thus, we may assume Conjecture~\ref{conj:upp-mld} holds for generalized klt pairs.

Let $(X,\Delta+M)$ be a generalized pair.
If $K_X+\Delta+M$ is not pseudo-effective, 
then a minimal model program with scaling of an ample divisor for $(X,\Delta+M)$ terminates with a Mori fiber space by~\cite[Lemma 4.4]{BZ16}.
Hence, from now on, we may assume
that $K_X+\Delta+M$ is pseudo-effective.
In particular ${\rm Bs}_{-}(K_X+\Delta+M)$ is a proper Zariski closed subset of $X$.

Let $(X,\Delta+M)$ be a generalized klt pair.
Assume that $K_X+\Delta+M$ is pseudo-effective.
By induction on the Picard rank, 
it suffices to prove that a sequence of flips for the MMP with scaling for $(X,\Delta+M)$ always terminates.
Let 
\[
(X,\Delta+M) 
\dashrightarrow 
(X_1,\Delta_1+M_1) 
\dashrightarrow 
(X_2,\Delta_2+M_2) 
\dashrightarrow 
\dots
\dashrightarrow 
(X_i,\Delta_i+M_i) 
\dashrightarrow 
\cdots 
\] 
be a sequence of flips for the MMP
with scaling of an ample divisor
for $(X,\Delta+M)$.
For each pair $i\leq j$, we denote by
$\phi_{i,j}\colon X_i\dashrightarrow X_j$ 
the corresponding birational contracion. 
For each $i\geq 1$, we denote by
\[
B_i := 
{\rm Bs}_{-}(K_{X_i}+\Delta_i+M_i) \subsetneq X_i,
\] 
the diminished base locus.
By Conjecture~\ref{conj:dim-base-gen}, the subset $B_i\subset X_i$ is a proper Zariski closed subset.
By Conjecture~\ref{conj:dim-base-gen} and Lemma~\ref{lem:MMP-vs-dim}, 
we know that for each $i$, there exists $j_i$ so that
\[
B_i \subset {\rm Ex}(\phi_{i,j}).
\] 
In particular, if we set 
\[
m_i:= {\rm mld}(X_i,\Delta_i+M_i;B_i).
\] 
Then, the sequence $m_i$, with $i\geq 1$, 
is an infinite sequence.
By Lemma~\ref{lem:mld-diminished} and Lemma~\ref{lem:gen-monotonicity}, 
we conclude that for each $i$,
there exists $j>i$ for which
\[
m_j>m_i. 
\] 
By Conjecture~\ref{conj:upp-mld}, we know that $m_i\leq A(n)$ for every $i$.
Hence, the sequence $\{m_i\}_{i\geq 1}$ converges to some real number $m_\infty \in (0,A(n)]$.
Furthermore, there exists an strictly increasing subsequence of $\{m_i\}_{i\geq 1}$ which converges to $m_\infty$.
We can find $i_0$ large enough and $\delta \in (0,1)$ for which
\begin{equation}\label{eq:mld-cont-i_0}  
m_\infty \in [m_{i_0}, m_{i_0}+\delta). 
\end{equation} 
Note that the same inclusion~\eqref{eq:mld-cont-i_0} holds for every $j\geq i_0$.
Set 
\[
\delta_0:= m_{i_0}-m_\infty  + \frac{\delta}{2}.
\] 
By Lemma~\ref{lem:finiteness-e+d}, for each $i$, there are only finitely many exceptional divisors
$E_1^i,E_2^i, \dots , E_{s_i}^i$
with center in $B_i$ 
and log discrepancy in the interval 
\[
[m_i,m_i+\delta_0). 
\] 
We claim that $|s_i|$ is bounded above for $i$ large enough.
Indeed, for every divisorial valuation
$E_j^i$, with $i\geq i_0$, we have that
\begin{equation}\label{eq:i0-1}
a_{E_j^i}(X_{i_0},\Delta_{i_0}+M_{i_0}) 
\in 
\left[ m_{i_0}, m_{i_0} +\frac{\delta}{2} 
\right)
\end{equation} 
and 
\begin{equation}\label{eq:i0-2}
c_{E_j^i}(X_{i_0})\subseteq B_{i_0}.
\end{equation} 
By Lemma~\ref{lem:finiteness-e+d}, 
there is a positive integer $N_0$,
so that there are at most $N_0$ exceptional divisorial valuations over
$X_i$ satisfying inclusion~\eqref{eq:i0-1}
and containment~\eqref{eq:i0-2}.
Hence, the order of the set of divisorial valuations induced by the exceptional divisors
\[
\left\{ 
E_j^i \mid 1\leq j \leq s_i 
\right\} 
\] 
is bounded above by $N_0$. 
For each $i\geq i_0$, we may assume that
$E_1^i$ computes the minimal log discrepancy
of $(X_i,\Delta_i+M_i;B_i)$.
Hence, for each $i\geq i_0$,
there are at most $N_0$
exceptional divisors $E$ over $X_i$
with 
\[
c_E(X_i)=c_{E_1^i}(X)
\]
and
\[
a_E(X_i,\Delta_i+M_i) \in [m_i,m_i+\delta_0). 
\] 
By Conjecture~\ref{conj:gen-boundedness} and Lemma~\ref{lem:generic-vs-general}, we conclude that
for a general closed point $x\in c_E(X_i)$ we have that
\[
|\pi_1^{\rm reg}(X_i,\Delta_i+M_i;x_i)|
\leq \rho:=\rho(n,N_0,\epsilon,\delta_0),
\] 
where $\epsilon=\frac{m_{i_0}}{2}$.
Hence, by Lemma~\ref{lem:controlled-p1}, for $i\geq i_0$, we have that 
\[
m_i:= 
a_{E_1^i}(X_i,\Delta_i+M_i) 
\in 
\zz_{>0}\left[ \frac{1}{q\rho} \right]
\cap (0,a(n)],
\] 
where $q$ is a positive integer that only depends on the initial model $(X,\Delta+M)$
and $\rho$ only depends on $n,N_0,\epsilon$ and $\delta_0$.
In particular, both $q$ and $\rho$ are independent of $i$.
This contradicts the fact that
the sequence
$\{ m_i\}_{i\geq 1}$ admits an infinite strictly increasing sequence.
Thus, after finitely many flips with scaling, we have that 
$B_j={\rm Bs}_{-}(K_{X_j}+\Delta_j+M_j)=\emptyset$.
This means that $K_{X_j}+\Delta_j+M_j$ is a nef divisor, so the minimal model program terminates.
\end{proof} 

To conclude this section, we will show
that Conjecture~\ref{conj:boundedness} in dimension $4$ implies the termination of flips in dimension $4$. 

\begin{theorem}
Let $(X,\Delta)$ be a $4$-dimensional klt pair.
If Conjecture~\ref{conj:boundedness} holds in dimension $4$, then 
any minimal model program for $(X,\Delta)$ terminates in finitely many steps.
\end{theorem}

\begin{proof}
If the pair $(X,\Delta)$ is canonical, then the statement follows from~\cite[Theorem 2]{Fuj04}.
As it is usual, it suffices to prove that any sequence of flips terminates.
Let 
\[
(X,\Delta) 
\dashrightarrow 
(X_1,\Delta_1) 
\dashrightarrow 
(X_2,\Delta_2) 
\dashrightarrow 
\dots
\dashrightarrow 
(X_i,\Delta_i) 
\dashrightarrow 
\cdots 
\] 
For each $i$, we denote by $N_i$ the number
of non-canonical exceptional divisors over $X_i$, with respect to $(X_i,\Delta_i)$.
By monotonicity, we have that 
$N_j \leq N_j$ for $j\geq i$.
If $N_i=0$ for some $i$, then the theorem follows from the canonical case~\cite[Theorem 2]{Fuj04}.
Hence, we may assume that $N=N_i$ for every $i$ large enough. 
Truncating the sequence of flips, we may simply assume that $N=N_i$ for every $i\geq 1$.
For each $i\geq 1$, we denote by 
$E_i^1,\dots,E_i^N$ the exceptional divisors over $(X_i,\Delta_i)$ with log discrepancy in the interval $(0,1)$.
We claim that, after finitely many flips, 
the flipping locus of $X_i\dashrightarrow X_{i+1}$ is disjoint from the generic points of the set
\begin{equation}\label{eq:union-centers}
\bigcup_{j=1}^N c_{E_j^i}(X_i). 
\end{equation} 
In other words, for each $j\in \{1,\dots, N\}$, the center of $E_i^j$ in $X_i$ is contained in the flipping locus
only finitely many times.
Up to reordering, we may assume that $E^1_i$ computes the minimal log discrepancy of $(X_i,\Delta_i)$ at the generic points of its center.
We know that $(X,\Delta)$ is $\epsilon$-log canonical for certain $\epsilon>0$.
By Lemma~\ref{lem:finiteness-e+d}, there exists at most $N'\geq N$ exceptional divisors over $X$ with log discrepancy in the interval $(\epsilon,1+\epsilon)$. 
Hence, for every model $X_i$, there are at most $N'$ exceptional divisors with center 
equal to $c_{E_1^i}$ and log discrepancy in the interval
\[
[ a_{E_1^i}(X_i,\Delta_i), 
a_{E_1^i}(X_i,\Delta_i) + \epsilon). 
\] 
By Conjeture~\ref{conj:boundedness}
and Lemma~\ref{lem:generic-vs-general},
we conclude that for a general closed point $x\in c_{E_1^i}(X_i)$, we have that 
\[
|\pi_1^{\rm reg}(X_i,\Delta_i+M_i;x_i)| 
\leq \rho := \rho(4,N',\epsilon,\epsilon).
\] 
Hence, by Lemma~\ref{lem:controlled-p1}, for $i\geq 1$, we have that
\[
a_{E_1^i} \in 
\zz_{>0}\left[ \frac{1}{q\rho} \right]
\cap (0,1).
\] 
Thus, the center of $E_1^i$ can only be flipped finitely many times by Lemma~\ref{lem:gen-monotonicity}.
Proceeding inductively with the other exceptional divisors $E_j^i$, we conclude that the flipping loci are eventually disjoint from the Zariski closed subset~\eqref{eq:union-centers}.
By truncating the sequence, we may assume that this condition happens for every single flip.
In particular, we will have that 
the log discrepancy
\[
a_{E_j^i}(X_i,\Delta_i)  \in (0,1)
\] 
is independent of $i$.
Hence, we can replace the sequence of flips for $(X,\Delta)$
with a sequence of flips for
its canonicalization $(Y,\Delta_Y)$.
The latter must terminate by~\cite[Theorem 2]{Fuj04}, so the former must terminate as well.
\end{proof}

\section{Examples}
\label{sec:ex}

In this section, we show that all the hypotheses of Conjecture~\ref{conj:boundedness} are actually necessary. 

\begin{example}
{\em  
In this example, we show that the control of $N$
is necessary for Conjecture~\ref{conj:boundedness}.
Let $\sigma_n \subset N_\qq \simeq \qq^2$ 
be the cone spanned by $(0,1)$ and $(n,1)$.
Then, we have that 
\[
X(\sigma_n) \simeq \{(x,y,z)\mid x^2+y^2+z^n=0\}. 
\]
Let $x\in X(\sigma_n)$ be the origin. 
By Proposition~\ref{prop:reg-toric}, we have that 
\[
\pi_1^{\rm reg}(X(\sigma_n)) 
\simeq \pi_1(X(\sigma_n)\setminus \{x\}) \simeq \zz_n. 
\] 
Observe that there are exactly $n-1$ divisors computing the minimal log discrepancy of $(X(\sigma_n);x)$.
By Proposition~\ref{prop:toric-mld}, we know that the minimal log discrepancy of $X(\sigma_n)$ at $x$ equals $1$.
}
\end{example}

\begin{example}
{\em 
In this example, we show that the control of $\epsilon_n$
is necessary in Conjeture~\ref{conj:boundedness}.
Let $\sigma_n \subset N_\qq\simeq \qq^2$
be the cone spanned by $(-1,n)$ and $(1,n)$.
Then, the toric singularity $(X(\sigma_n),x)$ is isomorphic
to the cone over the rational curve of degree $n$. 
There is a unique divisor computing the minimal log discrepancy.
By Proposition~\ref{prop:toric-mld}, we know that
\[
{\rm mld}(X(\sigma_n);x)= \frac{2}{n}.
\]
On the other hand, by Proposition~\ref{prop:reg-toric}, we have that 
\[
\pi_1^{\rm reg}(X(\sigma_n);x) \simeq \pi_1(X(\sigma_n)\setminus \{x\}) \simeq \zz_{2n}.
\] 
}
\end{example}

\begin{example}
{\em 
In this example, we show that the control of $\delta>0$ is necessary in Conjecture~\ref{conj:boundedness}.
This means that it is not only necessary to control the number of divisors computing the minimal log discrepancy,
but we also need to control the number of divisors hitting near the minimal log discrepancy.
Let $\sigma_r \subset N_\qq \simeq \qq^3$ be the cone spanned by the lattice elements $(1,0,0),
(0,1,0)$, and
$(1,1,r)$.
Then $(X(\sigma_r),x)$ is a terminal $3$-fold isolated singularity.
By Proposition~\ref{prop:reg-toric}, we have that 
\[
\pi_1^{\rm reg}(X(\sigma_r),x) \simeq \zz_r. 
\] 
On the other hand, by Proposition~\ref{prop:toric-mld}, there is a unique divisor computing the minimal log discrepancy which equals $1+\frac{1}{r}$.
Furthermore, for each $i\in \{0,\dots,r-1\}$, there is a divisor with log discrepancy $2-\frac{i}{r}$.
Hence, for every $\delta>0$, there are at least 
$\lfloor r\delta\rfloor$ divisors with log discrepancy in the interval
\[
\left[1+\frac{1}{r}, 1+\frac{1}{r}+\delta \right).
\] 
Hence, we need to give an upper bound $N$ for the number of 
divisors with log discrepancy in the interval $[m,m+\delta)$,
and not only for those computing the minimal log discrepancy.
}
\end{example}

\begin{example}
{\em 
Finally, we show that the control in the dimension is necessary.
To do so, we will use the notation of the proof of Proposition~\ref{prop:toric-case}.
We consider the lattice 
\[
\zz^n+\zz\left(\frac{1}{n},\dots ,\frac{1}{n}\right) \subset N_\zz \simeq \zz^n.
\]
Then, the corresponding toric singularity $X(\sigma_n)$ 
is canonical.
The minimal log discrepancy equals $1$ and there are no divisors
with log discrepancy in the interval $(1,2)$.
On the other hand, 
by Proposition~\ref{prop:reg-toric}, we have that 
\[
\pi_1^{\rm reg}(X(\sigma_n);x) \simeq \zz_n.
\] 
Hence, the control on the dimension $n$ is necessary for Conjecture~\ref{conj:boundedness}.
}
\end{example}

\bibliographystyle{habbrv}
\bibliography{bib}

\end{document}